\documentclass{amsart}

\usepackage{amsmath}
\usepackage{amsthm}
\usepackage{amsopn}
\usepackage{amssymb}
\usepackage[cmtip, all]{xy}
\usepackage{amsbsy}
\usepackage{stmaryrd}

\theoremstyle{definition}
\newtheorem{Lem}{Lemma}[section]

\newtheorem{Thm}[Lem]{Theorem}

\newtheorem{Rk}[Lem]{Remark}
\newtheorem{Def}[Lem]{Definition}

\DeclareMathAlphabet{\mathpzc}{OT1}{pzc}{m}{it}

\DeclareMathOperator*{\holim}{holim}

\DeclareMathOperator*{\colim}{colim}

\newcommand{\spt}{\mathrm{Spt}}
\newcommand{\cat}{\mathrm{Spt}_G}
\newcommand{\ccat}{c(\mathrm{Spt}_G)}
\newcommand{\typical}{{\overset{\textbf{\Large{.}}}{}}}
\newcommand{\zig}{\addtocounter{Lem}{1}\tag{\theLem}} 
\newcommand{\superc}{\scriptscriptstyle{C}} 
\newcommand{\fibrant}{{\scriptstyle{\mathpzc{Fib}}}}
 
\def\:{\colon}

\begin{document}
 
\title{Delta-discrete $G$-spectra and iterated \\homotopy fixed points}
\author{Daniel G. Davis$\sp 1$}
\address{Department of Mathematics, University of Louisiana at 
Lafayette, Lafayette, LA 70504, U.S.A.}
\email{dgdavis@louisiana.edu}
\begin{abstract}
Let $G$ be a profinite group with finite virtual cohomological dimension 
and let $X$ be a discrete $G$-spectrum. If $H$ and $K$ are closed subgroups of 
$G$, with $H \vartriangleleft K$, then, in general, 
the $K/H$-spectrum $X^{hH}$ is not known 
to be a continuous $K/H$-spectrum, so that it is not known (in general) 
how to define the iterated homotopy fixed point spectrum 
$(X^{hH})^{hK/H}$. To address this situation, we define 
homotopy fixed points for 
delta-discrete $G$-spectra and show that the setting of 
delta-discrete 
$G$-spectra gives a good framework within which to work. 
In particular, we show 
that by using delta-discrete $K/H$-spectra, there is 
always an 
iterated homotopy fixed point spectrum, 
denoted $(X^{hH})^{h_\delta K/H}$, 
and it is just $X^{hK}$. Additionally, we show that for any 
delta-discrete $G$-spectrum $Y$, $(Y^{h_\delta H})^{h_\delta K/H} 
\simeq Y^{h_\delta K}$. Furthermore, if $G$ is an {\em arbitrary} profinite 
group, there is a delta-discrete $G$-spectrum $X_\delta$ that is 
equivalent to $X$ and, though $X^{hH}$ is not even known in general 
to have a $K/H$-action, there is always an equivalence 
$\bigl((X_\delta)^{h_\delta H}\bigr)^{h_\delta K/H} \simeq 
(X_\delta)^{h_\delta K}.$ Therefore, delta-discrete 
$L$-spectra, by letting $L$ equal $H, K,$ and $K/H$, 
give a way of resolving undesired deficiencies in our 
understanding of homotopy fixed points for discrete $G$-spectra.
\end{abstract}
\maketitle 
\footnotetext[1]{The author 
was supported by 
a grant from the Louisiana Board of Regents Support Fund.}

\section{Introduction}
\par
Let $\mathrm{Spt}$ be the stable model category of Bousfield-Friedlander 
spectra of simplicial sets. Also, given a profinite group $G$, 
let $\mathrm{Spt}_G$ 
be the model category of discrete $G$-spectra, in which a morphism $f$ is a 
weak equivalence (cofibration) if and only if $f$ is a weak equivalence (cofibration) 
in $\mathrm{Spt}$ (see \cite[Section 3]{cts}). Given a fibrant replacement functor 
\[(-)_{fG} \: \cat \rightarrow \cat, \ \ \ X \mapsto X_{fG},\] so that there is a natural 
trivial 
cofibration $X \rightarrow X_{fG}$, with $X_{fG}$ fibrant, in $\cat$, the $G$-homotopy 
fixed point spectrum $X^{hG}$ is defined by 
\[X^{hG} = (X_{fG})^G.\]
\par
Let $H$ and $K$ be closed subgroups of $G$, with $H$ normal in $K$, and, 
as above, let $X$ 
be a discrete $G$-spectrum. Then $K/H$ is a profinite group, and 
it is reasonable to expect that $X^{hH}$ is some 
kind of a continuous $K/H$-spectrum, so that the iterated homotopy fixed point 
spectrum $(X^{hH})^{hK/H}$ can be formed. Additionally, one might 
expect $(X^{hH})^{hK/H}$ to just be 
$X^{hK}$: following Dwyer and Wilkerson (see 
\cite[pg. 434]{wilkerson}), when these two homotopy fixed point spectra 
are equivalent to each other, for all $H$, $K$, and $X$,  we say that homotopy 
fixed points for $G$ have the {\em transitivity property}.  
\par
Under hypotheses on $G$ and $X$ that are different from those 
above, 
there are various cases 
where the iterated homotopy fixed point spectrum is well-behaved 
along the lines suggested above. 
For example,  
by \cite[Lemma 10.5]{wilkerson}, when $G$ is any discrete group, with 
$N \vartriangleleft G$ and 
$\mathcal{X}$ 
any $G$-space, the iterated homotopy fixed point space 
$(\mathcal{X}^{hN})^{hG/N}$ is always defined and is just $\mathcal{X}^{hG}$. 
Similarly, by \cite[Theorem 7.2.3]{Rognes}, if 
$E$ is an $S$-module and $A \rightarrow B$ is a faithful 
$E$-local $G$-Galois extension 
of commutative $S$-algebras, where $G$ is a stably dualizable group, then if 
$N$ is an allowable normal subgroup of $G$, 
$(B^{hN})^{hG/N}$ is defined and is equivalent to $B^{hG}$. 
\par
Also, by 
\cite[Theorem 4.6, (4)]{quicklubin}, if $G$ is any profinite group and $P$ is a 
symmetric profinite $G$-spectrum, and if $\hat{P}^{hG}$ denotes the 
profinite homotopy fixed point spectrum (see \cite[Remark 4.2]{quicklubin}), 
then, if $N$ is any closed normal subgroup of $G$, there is a stable 
equivalence \[(\hat{P}^{hN})^{hG/N} \simeq \hat{P}^{hG}\] of profinite 
symmetric spectra. 
\par
Let $k$ be a spectrum such that 
the Bousfield localization $L_k(-)$ is equivalent to $L_ML_T(-),$ 
where $M$ is a finite spectrum and $T$ is smashing, and 
let $A$ be a $k$-local 
commutative $S$-algebra. If 
a spectrum $E$ is a consistent profaithful $k$-local 
profinite $G$-Galois extension of $A$ of finite vcd (the meaning of 
these terms is explained in \cite{joint}), 
then, by \cite[Proposition 7.1.4, Theorem 7.1.6]{joint}, 
\[(E^{h_kN})^{h_kG/N} \simeq E^{h_kG},\] 
for any closed normal subgroup $N$ of $G$, where, for example, 
$(-)^{h_kG}$ denotes 
the $k$-local homotopy fixed points (as defined in 
\cite[Section 6.1]{joint}). 
\par
% pg. 15 of fausk was helpful for the below
Finally, let $G$ be any compact Hausdorff group, let $R$ be an orthogonal 
ring spectrum satisfying the assumptions of 
\cite[Section 11.1, lines 1--3]{fausk}, and let $\mathcal{M}_R$ 
be the category of $R$-modules. Also, let $M$ be any pro-$G$-$R$-module 
(so that, for example, the pro-spectrum $M$ is a pro-$R$-module 
and a pro-orthogonal 
$G$-spectrum). By \cite[Proposition 11.5]{fausk}, 
if $N$ is any closed normal subgroup of $G$, 
then there is an equivalence \[(M^{h_GN})^{hG/N} \simeq M^{hG}\] 
of pro-spectra in the Postnikov 
model structure on the category of pro-objects in $\mathcal{M}_R$. 
Here, $M^{h_GN}$ is the $N$-$G$-homotopy fixed point pro-spectrum 
of \cite[Definition 11.3]{fausk} and, as discussed in \cite[pg.~165]{fausk}, 
there are cases when $M^{h_GN} \simeq M^{hN}$.
\par
The above results show that when one works with the hypotheses 
of $G$ is profinite and $X \in \mathrm{Spt}_G$ for the first time, 
it certainly is not unreasonable to hope that the expression 
$(X^{hH})^{hK/H}$ makes sense and that it fits into an equivalence 
$(X^{hH})^{hK/H} \simeq X^{hK}$. But it turns out that, 
in this setting, in general, 
these constructions are more subtle than the above 
results might suggest. For example, as 
explained in 
\cite[Section 5]{fibrantmodel}, $X^{hH}$ is not even 
known to be a $K/H$-spectrum. However, when $G$ has finite virtual 
cohomological dimension ({\em finite vcd}; that is, there 
exists an open subgroup $U$ and a positive integer $m$ such that the 
continuous cohomology $H^s_c(U;M) = 0$, whenever $s >m$ and $M$ is a discrete 
$U$-module), then \cite[Corollary 5.4]{fibrantmodel} shows that $X^{hH}$ is 
always weakly equivalent to a $K/H$-spectrum. 
\par
But, as explained 
in detail in 
\cite[Section 3]{iterated}, even when $G$ has finite vcd, it is not known, in general, 
how to 
view $X^{hH}$ as a continuous $K/H$-spectrum (in the sense of 
\cite{cts, joint}), so that it is not known how to form $(X^{hH})^{hK/H}$. 
For example, when $G = K = \mathbb{Z}/p \times \mathbb{Z}_q$, where 
$p$ and $q$ are distinct primes, and 
$H = \mathbb{Z}/p$, Ben Wieland found an example of a discrete $G$-spectrum 
$Y$ such that 
$Y^{hH}$ is not a discrete $K/H$-spectrum (see \cite[Appendix A]{iterated}). 
More generally, it is not known if $Y^{hH}$ is a continuous $K/H$-spectrum, and 
there is no known construction of $(Y^{hH})^{hK/H}$.
\par
By \cite[Section 4]{iterated}, if $G$ is any profinite group and $X$ is a hyperfibrant 
discrete $G$-spectrum, then $X^{hH}$ is always a discrete $K/H$-spectrum, and 
hence, $(X^{hH})^{hK/H}$ is always defined. However, it is not known if 
$(X^{hH})^{hK/H}$ must be equivalent to $X^{hK}$. Also, \cite[Section 4]{iterated} 
shows that if $X$ is a totally hyperfibrant discrete $G$-spectrum, then 
$(X^{hH})^{hK/H}$ is $X^{hK}$. But, as implied by our remarks above regarding 
$Y$, it is not known that all the objects in $\cat$ are hyperfibrant, let alone 
totally hyperfibrant.
\par
The above discussion makes it clear that there are nontrivial gaps in our 
understanding of iterated homotopy fixed points in the world of $\mathrm{Spt}_G$. 
To address these deficiencies,
% in our understanding of iterated homotopy fixed points, 
in this paper 
we define and study homotopy fixed points for 
{\em delta-discrete $G$-spectra}, $(-)^{h_\delta G}$, 
and within this framework, when $G$ has finite vcd 
and $X \in \cat$, 
we find that the iterated homotopy fixed point spectrum 
$(X^{hH})^{h_\delta K/H}$ is always defined and is equivalent to $X^{hK}$. 
In fact, when $G$ has finite vcd, 
if $Y$ is one of these delta-discrete $G$-spectra, 
then there is an equivalence
\[(Y^{h_\delta H})^{h_\delta K/H} \simeq Y^{h_\delta K}.\] 
\par
Before introducing this paper's approach to iterated homotopy 
fixed points in $\mathrm{Spt}_G$ in more detail, we quickly discuss 
some situations 
where the difficulties with iteration that were described earlier 
vanish, since 
it is helpful to better understand where the obstacles in 
``iteration theory" are.
\par
In general, for any profinite group $G$ and $X \in \mathrm{Spt}_G$, 
\[(X^{h\{e\}})^{hK/\{e\}} = \bigl((X_{fG})^{\{e\}}\bigr)^{hK} \simeq X^{hK}\] and, 
if $H$ is open in $K$, then $X_{fK}$ is fibrant in $\mathrm{Spt}_H$ and 
$(X_{fK})^H$ is fibrant in $\mathrm{Spt}_{K/H}$, so that \[(X^{hH})^{hK/H} 
= ((X_{fK})^H)^{hK/H} \simeq ((X_{fK})^H)^{K/H} = (X_{fK})^K
= X^{hK}\] (see \cite[Proposition 3.3.1]{joint} and \cite[Theorem 3.4]{iterated}). 
Thus, in general, the difficulties in forming the iterated homotopy 
fixed point spectrum occur only when $H$ is a nontrivial non-open 
(closed normal) subgroup of $K$.
\par
Now let $N$ be any nontrivial closed normal subgroup of 
$G$. 
There are cases where, thanks to a particular property that $G$ has, 
the spectrum $(X^{hN})^{hG/N}$ 
is defined, with 
\begin{equation}\label{thedesiredguy}\zig
(X^{hN})^{hG/N} \simeq X^{hG}.
\end{equation} 
To explain this, we assume that $G$ is infinite:
\begin{list}{}{\setlength{\leftmargin}{1.8\parindent} 
\setlength{\labelwidth}{\parindent}
\setlength{\rightmargin}{1.1\parindent}}
\item[$\bullet$]
if $G$ has finite cohomological dimension (that is, there exists a positive 
integer $m$ such that $H^s_c(G; M) = 0$, whenever $s > m$ and $M$ is a 
discrete $G$-module), then, by \cite[Corollary 3.5.6]{joint}, $X^{hN}$ is a 
discrete $G/N$-spectrum, so that $(X^{hN})^{hG/N}$ 
is defined, and, as hoped, the equivalence in (\ref{thedesiredguy}) 
holds;
\item[$\bullet$]
the profinite group $G$ is {\em just infinite} if $N$ always has 
finite index in $G$ (for more details about such 
groups, see \cite{justinfinite}; this interesting class of 
profinite groups includes, 
for example, the 
finitely generated pro-$p$ 
Nottingham group over the finite field $\mathbb{F}_{p^n}$, 
where $p$ is any prime and $n \geq 1$, and 
the just infinite profinite branch groups (see 
\cite{grigorchuk})), and thus, if $G$ is just infinite, then $N$ is always 
open in $G$, so that, as explained above, 
(\ref{thedesiredguy}) is valid; and
\item[$\bullet$] 
if $G$ has the property that every nontrivial 
closed subgroup is open, then, by \cite[Corollary 1]{oates} 
(see also \cite{dikran}), $G$ is topologically isomorphic 
to $\mathbb{Z}_p$, for some prime $p$, and, as before, 
(\ref{thedesiredguy}) holds. 
\end{list}
\par
In addition to the above cases, 
there is a family of 
examples in chromatic stable homotopy theory 
where iteration works in the desired way. 
Let $k$ be any finite field containing $\mathbb{F}_{p^n}$, where $p$ is any 
prime and $n$ is any positive integer. Given any height $n$ formal group 
law $\Gamma$ over $k$, let $E(k, \Gamma)$ be the Morava $E$-theory 
spectrum that satisfies 
\[\pi_\ast(E(k, \Gamma)) = W(k)\llbracket u_1, ..., u_{n-1}\rrbracket 
[u^{\pm 1}],\] 
where $W(k)$ denotes the Witt vectors, the 
degree of $u$ is $-2$, and 
the complete power series 
ring $W(k)\llbracket u_1, ..., u_{n-1} \rrbracket \cdot u^0$ 
is in degree zero (see \cite[Section 7]{Pgg/Hop0}). 
Also, let $G = S_n \rtimes \mathrm{Gal}(k/\mathbb{F}_p)$, 
the extended Morava stabilizer group: $G$ is a compact $p$-adic analytic group, 
and hence, has finite vcd, and, by \cite{Pgg/Hop0}, 
$G$ acts on $E(k, \Gamma)$. Then, 
by using \cite{DH, LHS, Rognes, joint} 
and the notion of total hyperfibrancy, 
\cite{iterated} shows that 
\[\bigl(E(k, \Gamma)^{hH}\bigr)^{hK/H} \simeq E(k,\Gamma)^{hK},\]
for all $H$ and $K$ defined as usual. Here, 
$E(k, \Gamma)$ is a continuous $G$-spectrum and not a discrete $G$-spectrum 
($\pi_0(E(k, \Gamma))$ is not a discrete $G$-module). If 
$F$ 
is any finite spectrum that is of type $n$, then, 
as in \cite[Corollary 6.5]{cts}, 
$E(k, \Gamma) \wedge F$ is a 
discrete $G$-spectrum and, again by the technique of \cite{iterated},
\[\bigl((E(k, \Gamma) \wedge F)^{hH}\bigr)^{hK/H} \simeq 
(E(k, \Gamma) \wedge F)^{hK}.\] 
We note that when $E(k, \Gamma) = E_n$, 
the Lubin-Tate spectrum, \cite[pg.~2883]{iterated} reviews 
some examples of how $\bigl(E(k, \Gamma)^{hH}\bigr)^{hK/H}$ 
plays a useful role in chromatic 
homotopy theory.
\par
Now we explain the approach of this paper to iterated homotopy 
fixed points in more detail. 
Let $G$ be an arbitrary profinite group and, as usual, let $X \in \cat$. Also, 
let $c(\cat)$ be the category of cosimplicial discrete $G$-spectra 
(that is, the category of cosimplicial objects in $\cat$). 
If $Z$ is a spectrum (which, in this paper, always 
means Bousfield-Friedlander spectrum), we let $Z_{k,l}$ denote the $l$-simplices 
of the $k$th simplicial set $Z_k$ of $Z$. Then 
$\mathrm{Map}_c(G,X)$ is the discrete 
$G$-spectrum that is defined by 
\[\mathrm{Map}_c(G,X)_{k,l} = \mathrm{Map}_c(G,X_{k,l}),\]
the set of 
continuous functions $G \rightarrow X_{k,l}$. 
The $G$-action on $\mathrm{Map}_c(G,X)$ 
is given by $(g \cdot f)(g') = f(g'g)$, for $g, g' \in G$ and $f \in \mathrm{Map}_c(G,X_{k,l}).$ 
As explained 
in \cite[Definition 7.1]{cts}, the functor 
\[\mathrm{Map}_c(G,-) \colon \mathrm{Spt}_G \rightarrow \mathrm{Spt}_G, \ \ \ 
X \mapsto \mathrm{Map}_c(G,X),\] 
forms a triple and there is 
a cosimplicial discrete $G$-spectrum $\mathrm{Map}_c(G^\bullet, X),$ 
where, for each $[n] \in \Delta$, the $n$-cosimplices satisfy the isomorphism
\[\mathrm{Map}_c(G^\bullet, X)^n \cong \mathrm{Map}_c(G^{n+1},X).\]  
\par
Following \cite[Remark 7.5]{cts}, let 
\[\widehat{X} = \colim_{N \vartriangleleft_o G} (X^N)_f,\] 
a filtered colimit over the open normal subgroups of $G$. Here, 
$(-)_f \: \spt \rightarrow \spt$ denotes a fibrant replacement functor 
for the model category 
$\spt$. Notice that $\widehat{X}$ is a discrete $G$-spectrum 
and a fibrant object in $\mathrm{Spt}$. 
Now let 
\[X_\delta = \holim_\Delta \mathrm{Map}_c(G^\bullet, \widehat{X}).\] 
(In the subscript of ``$X_\delta$," instead of ``$\Delta$," we use 
its less obtrusive and lowercase counterpart.) In the definition of 
$X_\delta$ (and everywhere else in this paper), the homotopy limit 
(as written in the definition) is 
formed in $\mathrm{Spt}$ (and not in $\mathrm{Spt}_G$). 
As explained in Section \ref{solution}, there is a natural 
$G$-equivariant map
\[\Psi \: X \overset{\simeq}{\longrightarrow} 
X_\delta\] that is a weak equivalence in $\spt$. Since 
$\widehat{X}$ is a fibrant spectrum, 
$\mathrm{Map}_c(G^\bullet, \widehat{X})^n$ is a fibrant spectrum 
(by applying \cite[Corollary 3.8, Lemma 3.10]{cts}), for each $[n] \in \Delta$. 
\par
The map $\Psi$ and the features of its target $X_\delta$ 
motivate the following definition. 
\begin{Def}\label{basic}
Let $X^\bullet$ be a cosimplicial discrete $G$-spectrum such that 
$X^n$ is fibrant in $\spt$, for each $[n] \in \Delta$. Then 
the $G$-spectrum \[\holim_\Delta X^\bullet\] is a {\em delta-discrete $G$-spectrum.} 
\end{Def}
\par
The weak equivalence 
$\Psi$ gives a natural way of associating a 
delta-discrete $G$-spectrum 
to every discrete $G$-spectrum. Also, we will see that 
$X_\delta$ plays a useful role in 
our work on iterated homotopy fixed points. 
\par
Now we give several more useful definitions, including 
a definition of weak equivalence 
for the setting of delta-discrete $G$-spectra.
\begin{Def}
Let $c(\mathrm{Spt})$ be the category of cosimplicial spectra. 
If $Z$ is a spectrum, let $\mathrm{cc}^\bullet(Z)$ denote 
the constant cosimplicial object (in $c(\mathrm{Spt})$) 
on $Z$. Also, given a discrete 
$G$-spectrum $X$, let
\[\mathrm{cc}_G(X) = \holim_\Delta \mathrm{cc}^\bullet(\widehat{X}).\] 
Notice that $\mathrm{cc}_G(X)$ is a delta-discrete $G$-spectrum.
\end{Def}
\begin{Def} 
If the morphism $f^\bullet \: X^\bullet \rightarrow Y^\bullet$ 
of cosimplicial discrete $G$-spectra is an {\em objectwise 
weak equivalence} (that is, $f^n \: X^n \rightarrow Y^n$ is a 
weak equivalence in $\cat$, for each $[n] \in \Delta$), 
such that the induced 
$G$-equivariant map \[f = \holim_\Delta f^\bullet \: 
\holim_\Delta X^\bullet \rightarrow \holim_\Delta Y^\bullet\] has source 
and target equal to delta-discrete $G$-spectra, then $f$ 
is a weak equivalence in 
$\spt$ (since all the $X^n$ and $Y^n$ are fibrant spectra). We call 
such a map $f$ a {\em weak equivalence of delta-discrete $G$-spectra}.
\end{Def}
\par
Now we define the key notion of homotopy 
fixed points for delta-discrete $G$-spectra. 
\par
\begin{Def}\label{cosimplicial}
Given a delta-discrete $G$-spectrum 
$\smash{\displaystyle{\holim_\Delta X^\bullet}}$, 
the {\em homotopy fixed point spectrum} 
$(\smash{\displaystyle{\holim_\Delta X^\bullet}})^{h_\delta G}$ is given 
by \[(\holim_\Delta X^\bullet)^{h_\delta G} = 
\holim_{[n] \in \Delta} \, (X^n)^{hG},\] 
where $(X^n)^{hG}$ is the homotopy fixed point spectrum of the discrete 
$G$-spectrum $X^n$. We 
use the notation ``\,$(-)^{h_\delta G}\,$" and the phrase 
``delta-discrete homotopy fixed points" to 
refer to the general operation of taking the homotopy fixed points of 
a delta-discrete $G$-spectrum. 
\end{Def}
\par
We show that the homotopy fixed points $(-)^{h_\delta G}$ 
for delta-discrete $G$-spectra 
have the following properties:
\begin{itemize}
\item[(a)]
when $G$ is a finite group, the homotopy fixed points of a delta-discrete $G$-spectrum 
agree with the usual notion of homotopy fixed points for a finite group;
\item[(b)]
for any profinite group $G$, 
the homotopy fixed points of delta-discrete $G$-spectra can be 
viewed as the total right derived functor of 
\[\lim_\Delta (-)^G \: \ccat \rightarrow \spt,\] where 
$\ccat$ has the injective model category structure (which is defined in 
Section \ref{properties}); 
\item[(c)]
given $X \in \cat$ and the delta-discrete $G$-spectrum $X_\delta$ associated to 
$X$, then, if $G$ has finite vcd, there is a weak equivalence 
\[X^{hL} \overset{\simeq}{\longrightarrow} 
(X_\delta)^{h_\delta L}\] in $\spt$, for every closed subgroup $L$ of $G$;
\item[(d)]\label{embedding}
more generally, if $G$ is any 
profinite group and $X \in \cat$, then there is a 
$G$-equivariant map 
$X \overset{\simeq}{\longrightarrow} \mathrm{cc}_G(X)$ that is a weak 
equivalence of spectra and a weak equivalence 
\[X^{hG} \overset{\simeq}{\longrightarrow} 
(\mathrm{cc}_G(X))^{h_\delta G}\mathrm{;}\]
\item[(e)]
if $f$ is a weak equivalence of delta-discrete $G$-spectra, then the 
induced map $f^{h_\delta G} = 
\holim_{[n] \in \Delta} (f^n)^{hG}$ is a 
weak equivalence in $\spt$ (since each 
$(f^n)^{hG}$ is a weak equivalence between 
fibrant spectra); and
\item[(f)]
if $G$ has finite vcd, with closed subgroups $H$ and $K$ such that 
$H \vartriangleleft K$, and if $Y$ is a delta-discrete $G$-spectrum, 
then $Y^{h_\delta H}$ is a delta-discrete $K/H$-spectrum, and 
\[(Y^{h_\delta H})^{h_\delta K/H} 
\simeq Y^{h_\delta K},\] so that $G$-homotopy fixed points 
of delta-discrete $G$-spectra have the transitivity 
property.
\end{itemize}
\par
In the above list of properties, 
(a) is Theorem \ref{inthefinitecase}, 
(b) is justified in Theorem \ref{derivedfunctor}, 
(c) is Lemma \ref{consistency}, (d) is verified right 
after the proof of Theorem \ref{inthefinitecase}, and (f) is 
obtained in 
Theorem \ref{transitivityfordeltadiscretes} (and the three 
paragraphs that precede it). 
\par
Notice that property (b) above 
shows that the {\em homotopy} fixed points 
of delta-discrete $G$-spectra are the total right derived functor of 
fixed points, in the appropriate sense. Also, (d) shows that, for any 
$G$ and any $X \in \cat$, the delta-discrete $G$-spectrum 
$\mathrm{cc}_G(X)$ 
is equivalent to $X$ and their homotopy fixed points are the same. Thus, 
the setting of delta-discrete $G$-spectra and the 
homotopy fixed points $(-)^{h_\delta G}$ 
includes and generalizes the category of discrete $G$-spectra and 
the homotopy fixed points $(-)^{hG}$. 
Therefore, properties (a) -- (f) show that the homotopy 
fixed points of a delta-discrete $G$-spectrum 
are a good notion that does indeed deserve to be called 
``homotopy fixed points."
\par
Now suppose that $G$ has finite vcd and, as usual, let $X \in \cat$. 
As above, let $H$ and $K$ be closed subgroups of $G$, with $H$ normal 
in $K$. In Lemma \ref{stepone}, we show that, by making a 
canonical identification, $X^{hH}$ is a 
delta-discrete $K/H$-spectrum. Thus, it is 
natural to form the iterated homotopy fixed point spectrum 
$(X^{hH})^{h_\delta K/H}$ and, by Theorem \ref{steptwo}, 
there is a weak equivalence 
\[X^{hK} \overset{\simeq}{\longrightarrow} (X^{hH})^{h_\delta K/H}.\] 
In this way, we show that when $G$ has finite vcd, 
by using delta-discrete $K/H$-spectra, there is a sense in 
which the iterated homotopy 
fixed point spectrum can always be formed and the transitivity 
property holds.   
\par
More generally, in 
Theorem \ref{deltaiteration}, we show that for 
{\em any} $G$, 
though it is not known if $X^{hH}$ always has a $K/H$-action (as 
mentioned earlier), 
there is an equivalence 
\begin{equation}\label{gladtohave}\zig
\bigl((X_\delta)^{h_\delta H}\bigr)^{h_\delta K/H} 
\simeq (X_\delta)^{h_\delta K}.
\end{equation}
Thus, 
for any $G$, by using $(-)_\delta$ and 
$(-)^{h_\delta H}$, the delta-discrete 
homotopy fixed points for discrete $G$-spectra are -- 
{\em in general} -- 
transitive. Also, there is a map 
\[\rho(X)_{_{\negthinspace H}} \: X^{hH} \rightarrow (X_\delta)^{h_\delta H}
\] that relates 
the ``discrete homotopy fixed points" $X^{hH}$ to 
$(X_\delta)^{h_\delta H}$, 
and this map is a weak equivalence whenever 
the map $\colim_{U \vartriangleleft_o H} \Psi^U$ 
is a weak equivalence (see Theorem \ref{rho} 
and the discussion that precedes it).
\par
For any $G$ and $X$, since the $G$-equivariant map 
$\Psi \: X \rightarrow X_\delta$ is a weak equivalence, 
$X$ can always be regarded as a delta-discrete 
$G$-spectrum, and thus, $X$ can be thought of as having 
two types of homotopy fixed points, $X^{hH}$ and 
$(X_\delta)^{h_\delta H}$, and, though its 
discrete homotopy fixed points $X^{hH}$ 
are not known to always be well-behaved with respect to iteration, 
there is a reasonable 
alternative, the delta-discrete homotopy fixed points 
$(X_\delta)^{h_\delta H}$, 
which, thanks to (\ref{gladtohave}), are always well-behaved. 
\par
As the reader might have noticed, given the work of \cite{iterated} (as 
discussed earlier) 
and that of this paper, when $G$ has finite vcd and $X$ is a 
hyperfibrant discrete $G$-spectrum, there are two different 
ways to define an iterated homotopy fixed point spectrum: 
as $(X^{hH})^{hK/H}$ and as $(X^{hH})^{h_\delta K/H}$. Though 
we are not able to show that these two objects are always equivalent, 
in Theorem \ref{comparison}, we show that if the canonical map 
$(X_{fK})^H \rightarrow ((X_{fK})_{fH})^H$ is a weak equivalence, 
then there is a weak equivalence 
$(X^{hH})^{hK/H} \overset{\simeq}{\longrightarrow} 
(X^{hH})^{h_\delta K/H}.$
\par
In the last section of this paper, Section \ref{modelcategories}, we show 
in two different, but interrelated ways that, for arbitrary $G$, 
the delta-discrete homotopy fixed point spectrum is always 
equivalent to a discrete homotopy fixed point spectrum. In 
each case, the equivalence is induced by a map 
between a 
discrete $G$-spectrum and a delta-discrete $G$-spectrum that 
need 
{\em not} be a weak equivalence. In Remark \ref{forgetfulfunctor}, 
we note several consequences of this observation for the 
categories $c(\cat)$ and $c(\mathrm{Spt})$, when each is 
equipped with the injective 
model structure.
\vspace{.05in}
\par
\noindent
\textbf{Acknowledgements.} With regard to the topic of homotopy limits, which 
appear frequently in this paper, I thank Kathryn Hess and 
Michael Shulman for helpful 
discussions and Chris Douglas for several useful comments. 
\section{Delta-discrete $G$-spectra and iterated homotopy fixed points}\label{solution}
Let $G$ be an arbitrary profinite group and $X$ a discrete 
$G$-spectrum. In this section, we review (from \cite{fibrantmodel}) the construction of a 
natural $G$-equivariant map 
$\Psi \: X \rightarrow X_\delta$ that is a weak equivalence in 
$\spt$, giving a 
natural way of associating 
a delta-discrete $G$-spectrum 
to each object of $\cat$. Also, we show that when $G$ has finite vcd, 
then by using the framework of delta-discrete $K/H$-spectra,  
there is a sense in which 
homotopy fixed points of discrete $G$-spectra satisfy transitivity.
\par 
Now let $G$ be any profinite group. There is a $G$-equivariant 
monomorphism $i \: X \rightarrow \mathrm{Map}_c(G,X)$ that is defined, on 
the level of sets, by $i(x)(g) = g \cdot x$, where $x \in X_{k,l}$ and 
$g \in G$. Notice that $i$ induces a $G$-equivariant map 
\[\widetilde{i}_X \: \holim_\Delta \mathrm{cc}^\bullet(X) 
\rightarrow \holim_\Delta \mathrm{Map}_c(G^\bullet, X).\]  
\par
There is a natural $G$-equivariant map 
\[\psi \: X \cong \colim_{N \vartriangleleft_o G} X^N \rightarrow 
\colim_{N \vartriangleleft_o G} (X^N)_f = \widehat{X}\] to the 
discrete $G$-spectrum $\widehat{X}$. 
We would like to know that $\psi$ is a weak equivalence in $\cat$; however, 
the validity of this is not obvious, since $\{X^N\}_{N \vartriangleleft_o G}$ is not known to 
be a diagram of fibrant spectra, and hence, we cannot use the fact that filtered colimits 
preserve weak equivalences between fibrant spectra. Nevertheless, the following 
lemma shows that it is still the case that $\psi$ is a weak equivalence in $\cat$ 
(the author stated this without proof in 
\cite[Remark 7.5]{cts} and \cite[Definition 3.4]{fibrantmodel}).
\begin{Lem}\label{trivialandhard}
If $X$ is a discrete $G$-spectrum, then $\psi \: X \rightarrow \widehat{X}$ 
is a weak equivalence in $\cat$.
\end{Lem}
\begin{proof} 
Since $X$ is a discrete $G$-spectrum, 
\[\mathrm{Hom}_G(-,X) \:  (G \negthinspace - \negthinspace \mathbf{Sets}_{df})^\mathrm{op} 
\rightarrow \mathrm{Spt}, \ \ \ C \mapsto \mathrm{Hom}_G(C,X)\] is a presheaf of 
spectra on the site 
$G \negthinspace - \negthinspace \mathbf{Sets}_{df}$ of finite discrete $G$-sets 
(for more detail, we refer the reader to \cite[Section 3]{cts} and 
\cite[Sections 2.3, 6.2]{Jardine}). Here, the $l$-simplices of the $k$th simplicial 
set $\mathrm{Hom}_G(C,X)_k$ are given by $\mathrm{Hom}_G(C,X_{k,l})$. Also, the 
composition
\[(-)_f \circ \mathrm{Hom}_G(-,X) \: 
(G \negthinspace - \negthinspace \mathbf{Sets}_{df})^\mathrm{op} 
\rightarrow \mathrm{Spt}, \ \ \ C \mapsto (\mathrm{Hom}_G(C,X))_f\] is a presheaf of 
spectra, and there is a map of presheaves 
\[\widehat{\psi} \: \mathrm{Hom}_G(-,X) \rightarrow (-)_f \circ 
\mathrm{Hom}_G(-,X)\] that comes 
from the natural transformation $\mathrm{id}_\spt \rightarrow (-)_f$. 
\par
Since $\widehat{\psi}(C) \:  \mathrm{Hom}_G(C,X) \rightarrow (\mathrm{Hom}_G(C,X))_f$ 
is a weak equivalence for every 
$C \in G \negthinspace - \negthinspace \mathbf{Sets}_{df}$, the map 
$\pi_t(\widehat{\psi})$ 
of presheaves is an isomorphism 
for every integer $t$. 
Therefore, for every integer $t$, 
$\widetilde{\pi}_t(\widehat{\psi})$, 
the map of sheaves associated to 
$\pi_t(\widehat{\psi})$, 
is an isomorphism, so that the map 
$\widehat{\psi}$ is a local stable equivalence, and 
hence, a stalkwise weak 
equivalence. Thus, the map 
$\colim_{N \vartriangleleft_o G} \widehat{\psi}(G/N)$ is a weak equivalence 
of spectra, and therefore, the map
\[\colim_{N \vartriangleleft_o G} X^N 
\cong \colim_{N \vartriangleleft_o G} \mathrm{Hom}_G(G/N, X)
\overset{\simeq}{\rightarrow} 
\colim_{N \vartriangleleft_o G} (\mathrm{Hom}_G(G/N, X))_f \cong 
\colim_{N \vartriangleleft_o G} (X^N)_f\] is a 
weak equivalence, giving the desired conclusion.
\end{proof}
\begin{Def}[{\cite[pg. 145]{fibrantmodel}}]
Given any profinite group $G$ and any $X \in \cat$, 
the composition 
\[X \overset{\psi}{\rightarrow} \widehat{X} \overset{\cong}{\rightarrow} 
\lim_\Delta \mathrm{cc}^\bullet(\widehat{X}) 
\overset{\xi}{\rightarrow} 
\holim_\Delta \mathrm{cc}^\bullet(\widehat{X}) 
\overset{{\widetilde{i}}_{\widehat{X}}}{\longrightarrow} 
\holim_\Delta \mathrm{Map}_c(G^\bullet, \widehat{X}) = X_\delta\] 
of natural maps, where the map $\xi$ is the usual one 
(for example, see \cite[Example 18.3.8, \negthinspace (2)]{Hirschhorn}), 
defines the natural $G$-equivariant map 
\[ \Psi \: X \rightarrow X_\delta\] of spectra.  
\end{Def}
\par
The next result was obtained in \cite[pg. 145]{fibrantmodel}; 
however, we give 
the proof below for completeness and because it is a key result.
\begin{Lem}
If $G$ is any profinite group and $X \in \cat$, then the natural 
map 
$\Psi \: X \overset{\simeq}{\longrightarrow} 
X_\delta$ is a weak equivalence of spectra. 
\end{Lem}
\begin{proof}
Following \cite[pg. 145]{fibrantmodel}, 
there is a homotopy spectral sequence 
\[E_2^{s,t} = \pi^s(\pi_t(\mathrm{Map}_c(G^\bullet, \widehat{X}))) 
\Rightarrow 
\pi_{t-s}(X_\delta).\] Let $\pi_t(\mathrm{Map}_c(G^\ast, \widehat{X}))$ be 
the canonical cochain complex associated to 
the cosimplicial abelian group 
$\pi_t(\mathrm{Map}_c(G^\bullet, \widehat{X}))$. 
As in \cite[proof of Theorem 7.4]{cts}, there is an exact sequence 
\[0 \rightarrow \pi_t(\widehat{X}) \rightarrow 
\mathrm{Map}_c(G^\ast, \pi_t(\widehat{X})) \cong 
\pi_t(\mathrm{Map}_c(G^\ast, \widehat{X})),\] so that 
$E_2^{0,t} \cong \pi_t(\widehat{X}) 
\cong \pi_t(X),$ where the last isomorphism is 
by Lemma \ref{trivialandhard}, 
and $E_2^{s,t} = 0$, when $s > 0.$ Thus, the above spectral sequence 
collapses, giving the desired result. 
\end{proof}
\par
Now we consider the iterated homotopy fixed 
points of discrete $G$-spectra 
by using the setting of delta-discrete $K/H$-spectra. 
Here, as in the Introduction, $H$ and $K$ are closed subgroups of 
$G$, with $H \vartriangleleft K$. We begin with a definition and some 
preliminary observations. 
\begin{Def}
If $Y^\bullet$ is a 
cosimplicial discrete $G$-spectrum, such that $Y^n$ 
is fibrant in $\cat$, for each 
$[n] \in \Delta$, 
then we call $Y^\bullet$ a {\em cosimplicial fibrant discrete $G$-spectrum}.
\end{Def}
\begin{Lem}[{\cite[proof of Theorem 3.5]{fibrantmodel}}]\label{fibrancy} 
If $G$ is any profinite group and $X \in \cat$, 
then the cosimplicial discrete $G$-spectrum 
$\mathrm{Map}_c(G^\bullet, \widehat{X})$ is a 
cosimplicial fibrant discrete $L$-spectrum, 
for every closed subgroup $L$ of $G$.
\end{Lem}
\par
For the rest of this section, 
we assume that $G$ has finite vcd and, as usual, 
$X$ is a discrete $G$-spectrum. 
As in Lemma \ref{fibrancy}, let 
$L$ be any closed subgroup of $G$. 
Then, by 
\cite[Remark 7.13]{cts} and 
\cite[Definition 5.1, Theorem 5.2]{fibrantmodel} 
(the latter citation 
sets the former on a stronger footing), 
there is an identification
\begin{equation}\label{identify}\zig
X^{hL} = \holim_\Delta \mathrm{Map}_c(G^\bullet, \widehat{X})^L.
\end{equation}
(Notice that, under this identification, $X^{hL} \cong (X_\delta)^L$.)
\begin{Lem}\label{consistency}
For each $L$, there is a weak equivalence $X^{hL} 
\overset{\simeq}{\longrightarrow} (X_\delta)^{h_\delta L}$.
\end{Lem}
\begin{proof}
By Lemma \ref{fibrancy}, the fibrant replacement map
\[\mathrm{Map}_c(G^\bullet, \widehat{X})^n \overset{\simeq}{\longrightarrow} 
(\mathrm{Map}_c(G^\bullet, \widehat{X})^n)_{fL}\] 
is a weak equivalence between fibrant objects in $\mathrm{Spt}_L$, 
for each $[n] \in \Delta$, 
so that 
\[(\mathrm{Map}_c(G^\bullet, \widehat{X})^n)^L
\overset{\simeq}{\longrightarrow} 
((\mathrm{Map}_c(G^\bullet, \widehat{X})^n)_{fL})^L = 
(\mathrm{Map}_c(G^\bullet, \widehat{X})^n)^{hL}\] 
is a weak equivalence between fibrant objects in $\spt$. 
Thus, there is a weak equivalence 
\[X^{hL} = \holim_\Delta 
\mathrm{Map}_c(G^\bullet, \widehat{X})^L 
\overset{\simeq}{\longrightarrow} 
(\holim_\Delta \mathrm{Map}_c(G^\bullet, \widehat{X}))^{h_\delta L} = 
(X_\delta)^{h_\delta L}.
\]
\end{proof}
\par
It will be useful to recall (see \cite[Lemma 4.6]{iterated}) 
that the functor 
\[(-)^H \: \mathrm{Spt}_K \rightarrow \mathrm{Spt}_{K/H}, \ \ \ 
Y \mapsto Y^H\] preserves fibrant objects.
\par
Now we are ready for the task at hand. Since 
$\mathrm{Map}_c(G^\bullet, \widehat{X})^H$ is a cosimplicial 
discrete $K/H$-spectrum that is fibrant in $\mathrm{Spt}$ in 
each codegree, we can immediately conclude the following.
\begin{Lem}\label{stepone}
The spectrum $X^{hH} = 
\smash{\displaystyle{\holim_{\Delta}}} 
\, \mathrm{Map}_c(G^\bullet, \widehat{X})^H$
is a delta-discrete $K/H$-spectrum.
\end{Lem} 
\par
Lemma \ref{stepone} implies that 
\begin{align*}
(X^{hH})^{h_\delta K/H} 
& = 
\holim_{[n] \in \Delta} \bigl((\mathrm{Map}_c(G^\bullet, 
\widehat{X})^n)^H\bigr)^{hK/H}.
\end{align*}
\begin{Thm}\label{steptwo}
There is a weak equivalence $X^{hK} \overset{\simeq}{\longrightarrow} 
(X^{hH})^{h_\delta K/H}.$
\end{Thm}
\begin{proof}
Since $\mathrm{Map}_c(G^\bullet, \widehat{X})$ is a cosimplicial fibrant 
discrete $K$-spectrum, 
the diagram 
$\mathrm{Map}_c(G^\bullet, \widehat{X})^H$ is a 
cosimplicial fibrant discrete 
$K/H$-spectrum. Hence, each fibrant replacement map 
\[(\mathrm{Map}_c(G^\bullet, \widehat{X})^n)^H 
\overset{\simeq}{\longrightarrow}
((\mathrm{Map}_c(G^\bullet, \widehat{X})^n)^H)_{fK/H}\] is a weak equivalence between 
fibrant objects in $\mathrm{Spt}_{K/H}$, so that 
the induced map
\begin{equation}\label{moreprogress}\zig
\holim_{[n] \in \Delta} ((\mathrm{Map}_c(G^\bullet, \widehat{X})^n)^H)^{K/H} \overset{\simeq}{\longrightarrow}
\holim_{[n] \in \Delta} 
\bigl(((\mathrm{Map}_c(G^\bullet, \widehat{X})^n)^H)_{fK/H}\bigr)^{K/H}
\end{equation} is a weak equivalence. The weak equivalence in (\ref{moreprogress}) 
is exactly the weak equivalence
\begin{align*}
X^{hK} = 
\holim_{[n] \in \Delta} (\mathrm{Map}_c(G^\bullet, \widehat{X})^n)^K \overset{\simeq}{\rightarrow} 
\holim_{[n] \in \Delta} 
\bigl((\mathrm{Map}_c(G^\bullet,\widehat{X})^n)^H\bigr)^{hK/H} 
= (X^{hH})^{h_\delta K/H}.
\end{align*}  
\end{proof}
\par
Lemma \ref{stepone} and Theorem \ref{steptwo} show that, 
when $G$ has 
finite vcd, 
the iterated homotopy fixed point spectrum $(X^{hH})^{h_\delta K/H}$ 
is always defined and it is just $X^{hK}$.
\section{The homotopy fixed points of delta-discrete 
$G$-spectra\\as a total right derived 
functor}\label{properties}
\par
Let $G$ be any profinite group. 
In this section, we show that $(-)^{h_\delta G}$, the operation 
of taking the homotopy 
fixed points of a delta-discrete $G$-spectrum, can be viewed as 
the total right derived functor of 
\[\lim_\Delta (-)^G \: c(\mathrm{Spt}_G) \rightarrow 
\mathrm{Spt}, \ \ \ X^\bullet \mapsto \lim_{[n] \in \Delta} (X^n)^G,\] where 
$c(\mathrm{Spt}_G)$ 
has the injective model category structure (defined below). We 
will obtain this result as a special case of a more general result. 
Thus, we let $\mathcal{C}$ denote any small category and we use 
``$\,Z^\typical \,$" and ``$\,X^\typical \,$" 
% the following are options, but they don't work:
% $Z^{\boldsymbol{\cdot}}$ 
% $Z^{\textrm{\fontsize{2pt}{3pt}\selectfont \tiny{\textbullet}}}$
% $Z\textrm{\textsuperscript{\fontsize{1pt}{2pt}\selectfont \tiny{\textbullet}}}$
% $Z\textbf{\textsuperscript{\LARGE{.}}}$
%  $Z^{{\scriptscriptstyle \bullet}}$
% $Z^{\overset{\scriptscriptstyle \bullet}{}}$ 
to denote 
$\mathcal{C}$-shaped diagrams in $\mathrm{Spt}$ and $\cat$, 
respectively, since 
we are especially interested in the case when $\mathcal{C} = \Delta$.
\par
Recall that $\mathrm{Spt}$ is a combinatorial model category 
(for the definition of this notion,
we refer the reader to 
the helpful expositions in \cite[Section 2]{combinatorial} and 
\cite[Section A.2.6 (and Definition A.1.1.2)]{luriebook}); 
this well-known 
fact is stated explicitly in \cite[pg. 459]{rosickybrown}.
Therefore, \cite[Proposition A.2.8.2]{luriebook} implies that 
$\mathrm{Spt}^\mathcal{C}$, the category of functors $\mathcal{C} 
\rightarrow \mathrm{Spt}$, has an injective model category structure: 
more precisely, $\mathrm{Spt}^\mathcal{C}$ has a model structure in 
which 
a map $f^\typical \: 
Z^\typical \rightarrow W^\typical$ in $\mathrm{Spt}^\mathcal{C}$ 
is a weak equivalence (cofibration) 
if and only if, for each $C \in \mathcal{C},$ 
the map $f^{\superc} \: Z^{\superc} \rightarrow W^{\superc}$ is a 
weak equivalence (cofibration) in $\mathrm{Spt}$.
\par
Similar comments apply to $\mathrm{Spt}_G$. For example, the 
proof of \cite[Theorem 2.2.1]{joint} applies 
\cite[Definition 3.3]{hoveygeneral} to obtain the model structure 
on $\cat$, and thus, $\cat$ is a cellular model category. Hence, 
$\cat$ is cofibrantly generated. Also, the category 
$\cat$ is equivalent to the 
category of sheaves of spectra on the site 
$G \negthinspace - \negthinspace \mathbf{Sets}_{df}$ 
(see \cite[Section 3]{cts}), and thus, $\cat$ is a 
locally presentable category, since 
standard arguments show that such a category of sheaves is 
locally presentable. (For example, see the general comment
about such categories in \cite[pg. 2]{hagtwo}. The basic ideas are 
contained in the proof of the fact that a Grothendieck topos is 
locally presentable (see, for example, \cite[Proposition 3.4.16]{Borceux}), 
and hence, the category of sheaves of sets on the aforementioned 
site is locally presentable.) 
\par
From the above considerations, we conclude 
that $\cat$ is a combinatorial model category. 
% i verified that \cite{hagtwo} definitely refers to sheaves of 
% 	*symmetric* spectra as a combinatorial model category
Therefore, as before, 
\cite[Proposition A.2.8.2]{luriebook} 
implies that the category 
$(\mathrm{Spt}_G)^\mathcal{C}$ 
of $\mathcal{C}$-shaped diagrams in 
$\cat$ has an injective model structure in which a map $h^\typical$ is 
a weak equivalence (cofibration) if and only if 
each map $h^{\superc}$ is a weak equivalence (cofibration) in 
$\mathrm{Spt}_G$. 
\par
It will be useful to note that, by \cite[Remark A.2.8.5]{luriebook}, if 
$f^\typical$ is a fibration in $\mathrm{Spt}^\mathcal{C}$, then 
$f^{\superc}$ is a fibration in $\mathrm{Spt}$, 
for every $C \in \mathcal{C}$. Similarly, if $h^\typical$ is a 
fibration in $(\mathrm{Spt}_G)^\mathcal{C}$, then 
each map $h^{\superc}$ is a fibration in $\cat$.
\par
The functor 
\[\lim_\mathcal{C} (-)^G \: 
(\mathrm{Spt}_G)^\mathcal{C} 
\rightarrow \mathrm{Spt}, 
\ \ \ X^\typical \mapsto 
\lim_{\scriptscriptstyle{C} \in \mathcal{C}} 
(X^{\scriptscriptstyle{C}})^G\] is right adjoint to the functor 
$\underline{\mathrm{t}} \: \mathrm{Spt} 
\rightarrow (\mathrm{Spt}_G)^\mathcal{C}$ that sends 
an arbitrary spectrum $Z$ to 
the constant $\mathcal{C}$-shaped 
diagram $\underline{\mathrm{t}}(Z)$ 
on the discrete $G$-spectrum $t(Z)$, where 
(as in \cite[Corollary 3.9]{cts}) 
\[t \: \mathrm{Spt} \rightarrow \cat, \ \ \ Z \mapsto t(Z) =Z\] is the 
functor that equips $Z$ with the trivial $G$-action. 
If $f$ is a weak equivalence (cofibration) in $\mathrm{Spt}$, then $t(f)$ 
is a weak 
equivalence (cofibration) in $\mathrm{Spt}_G$, and hence, 
$\underline{\mathrm{t}}(f)$ is a weak equivalence (cofibration) in 
$(\cat)^\mathcal{C}$. This observation immediately gives the 
following result.
\begin{Lem}\label{quillen}
The functors $(\underline{\mathrm{t}}, \lim_\mathcal{C} (-)^G)$ 
are a Quillen pair for 
$(\mathrm{Spt}, (\cat)^\mathcal{C})$. 
\end{Lem}
\par
We let \[(-)_\fibrant \: (\cat)^\mathcal{C} \rightarrow (\cat)^\mathcal{C}, 
\ \ \ X^\typical \mapsto (X^\typical)_\fibrant\] 
denote a fibrant replacement functor, such 
that there is a morphism 
$X^\typical 
\rightarrow (X^\typical)_\fibrant$ in $(\cat)^\mathcal{C}$ that 
is a natural trivial cofibration, with 
$(X^\typical)_\fibrant$ fibrant, in $(\cat)^\mathcal{C}$. 
Then Lemma \ref{quillen} 
implies the existence of the total right derived functor 
\[\mathbf{R}(\lim_\mathcal{C} (-)^G) \: 
\mathrm{Ho}\bigl((\cat)^\mathcal{C}\bigr) \rightarrow 
\mathrm{Ho}(\mathrm{Spt}), \ \ \ X^\typical
\mapsto \lim_\mathcal{C} 
((X^\typical)_{\scriptstyle{\mathpzc{Fib}}})^G.\]   
\par
Now we prove the key result that will allow us to relate 
Definition \ref{cosimplicial} 
to the total right derived functor $\mathbf{R}(\lim_\Delta (-)^G)$.

\begin{Thm}\label{derived}
Given $X^\typical$ in $(\cat)^\mathcal{C}$, the canonical map 
\[\bigl(\mathbf{R}(\lim_\mathcal{C} (-)^G)\bigr)(X^\typical) = 
\lim_\mathcal{C} ((X^\typical)_\fibrant)^G 
\overset{\simeq}{\longrightarrow} 
\holim_\mathcal{C} ((X^\typical)_\fibrant)^G\] 
is a weak equivalence of spectra.
\end{Thm}

\begin{proof}
Let $t^\mathcal{C} \: 
\mathrm{Spt}^\mathcal{C} \rightarrow (\cat)^\mathcal{C}$ 
be the functor that sends $Z^\typical$ 
to $Z^\typical$, where each $Z^{\superc}$ is regarded as having the trivial $G$-action. Then 
$t^\mathcal{C}$ preserves weak equivalences and cofibrations, 
so that the 
right adjoint of $t^\mathcal{C}$, the fixed points functor 
$(-)^G \: (\cat)^\mathcal{C} \rightarrow \mathrm{Spt}^\mathcal{C}$, 
preserves 
fibrant objects. Thus, the diagram $((X^\typical)_\fibrant)^G$ is fibrant in 
$\mathrm{Spt}^\mathcal{C}$. 

\par
Let $Z^\typical$ be a fibrant object 
in $\mathrm{Spt}^\mathcal{C}$. Then, to 
complete the proof, it suffices to show that 
% i checked in two different ways that \lambda has not been used before 
%	in this paper
the following canonical map 
is a weak equivalence: \[\lambda \: \lim_\mathcal{C} Z^\typical 
\rightarrow \holim_\mathcal{C} Z^\typical.\] (Though 
this assertion seems to be well-known, for the sake of 
completeness, we justify it below. Also, see Remark 
\ref{holimremark}.) Since the 
functor $\mathrm{Spt} \rightarrow \mathrm{Spt}^\mathcal{C}$ 
that sends a spectrum $W$ 
to the constant $\mathcal{C}$-shaped 
diagram on $W$ preserves weak equivalences and cofibrations, 
its right adjoint, the functor 
$\lim_\mathcal{C} (-) \: \mathrm{Spt}^\mathcal{C} 
\rightarrow \mathrm{Spt}$, preserves 
fibrant objects. Thus, $\lim_\mathcal{C} Z^\typical$, the source of 
$\lambda$, is a fibrant spectrum. 
\par
Since 
$Z^\typical$ is fibrant in $\mathrm{Spt}^\mathcal{C}$, 
$Z^{\superc}$ is fibrant in $\mathrm{Spt}$, for each $C \in \mathcal{C}$, 
and hence, 
$\holim_\mathcal{C} Z^\typical$ is a 
fibrant spectrum. Therefore, 
since the source and target of $\lambda$ 
are fibrant spectra, to verify that $\lambda$ 
is a weak equivalence, we only have to show that 
the map 
\[\lambda_k \: 
\lim_\mathcal{C} (Z^\typical)_k = 
(\lim_\mathcal{C} Z^\typical)_k \rightarrow 
(\holim_\mathcal{C} Z^\typical)_k = \holim_\mathcal{C} (Z^\typical)_k\]
is a weak equivalence of simplicial sets, for each $k \geq 0$, where 
the limit and homotopy limit of $(Z^\typical)_k$ 
in the source and target, respectively, of $\lambda_k$ are formed 
in $\mathcal{S}$, the category of simplicial sets. To do this, we equip 
$\mathcal{S}^\mathcal{C}$, the 
category of $\mathcal{C}$-shaped diagrams of 
simplicial sets, with an injective model structure, so that 
a morphism 
$f^\typical \: K^\typical \rightarrow L^\typical$ in 
$\mathcal{S}^\mathcal{C}$ 
is a weak equivalence 
(cofibration) if and only if $f^{\superc} \: K^{\superc} \rightarrow 
L^{\superc}$ is a weak equivalence (cofibration) 
of simplicial sets, for each $C \in \mathcal{C}$ (the injective model 
structure on $\mathcal{S}^\mathcal{C}$ exists, 
for example, by \cite[pg. 403, Proposition 2.4]{GJ}). 
\par
Notice that the category $\mathrm{Spt}^\mathcal{C}$ is equal to 
the category $\mathrm{Spt}^{(\mathcal{C}^\mathrm{op})^\mathrm{op}}$ 
of presheaves of spectra on $\mathcal{C}^\mathrm{op}$, where 
$\mathcal{C}^\mathrm{op}$ is regarded as a site by equipping it 
with the chaotic topology. Then, by \cite[Remark 2.36]{Jardine}, 
the injective model structure on $\mathrm{Spt}^\mathcal{C}$ is 
exactly the local injective model structure (of \cite{Jardinecanada}) on 
the category 
$\mathrm{Spt}^{(\mathcal{C}^\mathrm{op})^\mathrm{op}}$ 
of presheaves 
of spectra on the site $\mathcal{C}^\mathrm{op}$. Thus, 
by \cite[Remark 2.35]{Jardine}, $(Z^\typical)_k$ is a globally 
fibrant simplicial presheaf of sets on the site 
$\mathcal{C}^\mathrm{op}$, which means exactly that 
$(Z^\typical)_k$ is fibrant in $\mathcal{S}^\mathcal{C}$. This 
conclusion implies that $\lambda_k$ is a weak equivalence, by 
the proof of \cite[pg. 407, Lemma 2.11]{GJ}, giving the desired 
result.
\end{proof} 
\par
Given $X^\typical \in (\cat)^\mathcal{C}$, we let 
$((X^\typical)_\fibrant)^{\superc}$ denote the value of 
$(X^\typical)_\fibrant \: \mathcal{C} \rightarrow \cat$ on the object 
$C \in \mathcal{C}$.
\begin{Lem}\label{generalC}
If $X^\typical$ is an object in $(\cat)^\mathcal{C}$, then 
there is a weak equivalence
\[\holim_{\superc \in \mathcal{C}} (X^{\superc})^{hG} 
\overset{\simeq}{\longrightarrow} \holim_{\superc \in \mathcal{C}} 
(((X^\typical)_\fibrant)^{\superc})^{G}.\] 
\end{Lem}
\begin{proof}
Let $(X^\typical)_{fG}$ be the object in $(\cat)^\mathcal{C}$ that 
is equal to the composition of functors 
\[ (-)_{fG} \circ (X^\typical) \: \mathcal{C} \rightarrow \cat, \ \ \ 
C \mapsto (X^{\superc})_{fG}.\] 
Since $X^\typical \rightarrow (X^\typical)_{fG}$ is a 
trivial cofibration in $(\cat)^\mathcal{C}$, the fibrant object 
$(X^\typical)_\fibrant$ induces a weak equivalence 
\[\ell^\typical \: (X^\typical)_{fG} \overset{\simeq}{\longrightarrow} 
(X^\typical)_\fibrant\] in 
$(\cat)^\mathcal{C}$. Therefore, since 
$(X^{\superc})_{fG}$ and 
$((X^\typical)_\fibrant)^{\superc}$ 
are fibrant discrete $G$-spectra, for each $C \in \mathcal{C}$,
there is a weak equivalence
\[\holim_{\superc \in \mathcal{C}} (\ell^{\superc})^G \:
\holim_{\superc \in \mathcal{C}} (X^{\superc})^{hG} = 
\holim_{\superc \in \mathcal{C}} ((X^{\superc})_{fG})^G 
\overset{\simeq}{\longrightarrow} \holim_{\superc \in \mathcal{C}} 
(((X^\typical)_\fibrant)^{\superc})^G.\] 
\end{proof}
\par
By letting $\mathcal{C} = \Delta$, Theorem \ref{derived} and 
Lemma \ref{generalC} immediately yield the next result, which allows 
us to conclude that {\em homotopy} 
fixed points for delta-discrete $G$-spectra, 
$(-)^{h_\delta G}$, can indeed be regarded as the total right derived 
functor of fixed points, in the appropriate sense.
\begin{Thm}\label{derivedfunctor}
If $\holim_\Delta X^\bullet$ is a delta-discrete $G$-spectrum, 
then there is a zigzag 
\[(\holim_\Delta X^\bullet)^{h_\delta G} 
\overset{\simeq}{\longrightarrow} \holim_\Delta
((X^\bullet)_\fibrant)^{G} 
\overset{\simeq}{\longleftarrow} 
\bigl(\mathbf{R}(\lim_\Delta (-)^G)\bigr)(X^\bullet)\] of 
weak equivalences in $\mathrm{Spt}$.
\end{Thm}
\section{Several properties of the homotopy fixed points\\of 
delta-discrete $G$-spectra}\label{basicproperties}
\par
Suppose that $P$ is a finite group and let $Z$ be a $P$-spectrum. 
Recall (for example, from \cite[Section 5]{cts}) that if $Z'$ is a 
$P$-spectrum and a fibrant object in $\mathrm{Spt}$, with a 
map $Z \overset{\simeq}{\rightarrow} Z'$ that is 
$P$-equivariant and a weak 
equivalence in $\mathrm{Spt}$, 
then $Z^{h'P}$, the usual homotopy fixed point spectrum 
$\mathrm{Map}_P(EP_+, Z')$ in the case when 
$P$ is a finite discrete group, 
can also be defined as  
\begin{equation}\zig\label{classical}
Z^{h'P} = \holim_P Z'.
\end{equation}  
Then the 
following result shows that the homotopy fixed points $(-)^{h_\delta G}$ 
of Definition \ref{cosimplicial} agree with those of (\ref{classical}), 
when the profinite group $G$ is finite and discrete. 
\begin{Thm}\label{inthefinitecase}
Let $G$ be a finite discrete group and let 
$\holim_\Delta X^\bullet$ be a delta-discrete $G$-spectrum 
(that is, $X^\bullet$ is a cosimplicial $G$-spectrum, with each 
$X^n$ 
a fibrant spectrum). Then there is a weak equivalence 
\[(\holim_\Delta X^\bullet)^{h_\delta G} 
\overset{\simeq}{\longrightarrow} 
(\holim_\Delta X^\bullet)^{h'G}.\]
\end{Thm}
\begin{proof} 
Given $[n] \in \Delta$, by \cite[Proposition 6.39]{Jardine}, 
the canonical map 
\[(X^n)^{hG} = ((X^n)_{fG})^G \cong \lim_G (X^n)_{fG} 
\overset{\simeq}{\longrightarrow} 
\holim_G (X^n)_{fG}\] is a weak equivalence. 
Also, notice that the target 
(since $(X^n)_{fG}$ is a fibrant spectrum, 
by \cite[Lemma 3.10]{cts}) 
and the source 
of this weak equivalence are fibrant spectra. 
Thus, there is a weak equivalence
\[
(\holim_\Delta X^\bullet)^{h_\delta G} = 
\holim_{[n] \in \Delta} (X^n)^{hG}  
\overset{\simeq}{\longrightarrow} 
\holim_{[n] \in \Delta} \holim_G (X^n)_{fG} 
\cong \holim_G \holim_{[n] \in \Delta} (X^n)_{fG}.\] 
The proof is finished by noting that 
\[\holim_G \holim_{[n] \in \Delta} (X^n)_{fG} = 
(\holim_\Delta X^\bullet)^{h'G};\] this 
equality, which is an application of (\ref{classical}), is due to the fact that 
the map 
$\holim_\Delta X^\bullet \rightarrow \holim_{[n] \in \Delta} (X^n)_{fG}$ 
is $G$-equivariant and a weak equivalence (since each 
map $X^n \rightarrow (X^n)_{fG}$ is a weak equivalence 
between fibrant objects in $\mathrm{Spt}$), 
with target a fibrant spectrum.
\end{proof}
\par
Now let $G$ be 
any profinite group and let $X$ be a discrete $G$-spectrum. We will 
show that there is a $G$-equivariant map $X \rightarrow 
\mathrm{cc}_G(X)$ 
that is a weak equivalence of spectra, along with a weak 
equivalence $X^{hG} \rightarrow (\mathrm{cc}_G(X))^{h_\delta G}$. 
Since these weak equivalences exist for any $G$ and all $X \in \cat$, 
we can think of the world of delta-discrete $G$-spectra as being 
a generalization of the category $\cat$.
\par
Given any spectrum $Z$, 
it is not hard to see that 
there is an isomorphism 
\[\mathrm{Tot}(\mathrm{cc}^\bullet(Z)) \cong Z;\] this was noted, 
for example, in the setting of simplicial sets, in \cite[Section 1]{neisen} and 
is verified for an arbitrary simplicial model category in 
\cite[Remark B.16]{galoishess}. 
Since the Reedy category $\Delta$ has fibrant constants (see 
\cite[Corollary 15.10.5]{Hirschhorn}), the canonical map 
$\mathrm{Tot}(\mathrm{cc}^\bullet(Z)) \rightarrow 
\holim_\Delta \mathrm{cc}^\bullet(Z)$ is a weak equivalence, 
whenever $Z$ is a fibrant spectrum, by 
\cite[Theorem 18.7.4, (2)]{Hirschhorn}. Thus, if $Z$ is a fibrant 
spectrum, there is a weak equivalence 
\[\phi_Z \: Z \cong \mathrm{Tot}(\mathrm{cc}^\bullet(Z)) 
\overset{\simeq}{\longrightarrow} 
\holim_\Delta \mathrm{cc}^\bullet(Z)\] in $\mathrm{Spt}$. In 
particular, since the discrete $G$-spectrum 
$\widehat{X}$ is a fibrant spectrum, the 
map $\phi_{\widehat{X}}$ is a weak equivalence. Therefore, 
since the map $\psi \: X \rightarrow \widehat{X}$ is a weak 
equivalence in $\cat$ (by Lemma \ref{trivialandhard}), the 
$G$-equivariant map
\[\phi_{\widehat{X}} \circ \psi \: 
X \overset{\simeq}{\longrightarrow} \widehat{X} 
\overset{\simeq}{\longrightarrow} 
\holim_\Delta \mathrm{cc}^\bullet(\widehat{X}) = 
\mathrm{cc}_G(X)\] and 
the map 
\[\phi_{(\widehat{X})^{hG}} \circ (\psi)^{hG} \: X^{hG} 
\overset{\simeq}{\longrightarrow} (\widehat{X})^{hG} 
\overset{\simeq}{\longrightarrow} 
\holim_\Delta \mathrm{cc}^\bullet((\widehat{X})^{hG}) 
= (\mathrm{cc}_G(X))^{h_\delta G}\] are weak equivalences 
(the map $\phi_{(\widehat{X})^{hG}}$ is a weak equivalence because 
$(\widehat{X})^{hG}$ 
is a fibrant spectrum).
% i verified that this paper uses the symbol "\phi" nowhere else
\section{Iterated homotopy fixed points for delta-discrete $G$-spectra}
\par
Throughout this section (except in Convention 5.1), we assume that 
the profinite group $G$ has finite vcd. 
We will show that $G$-homotopy fixed points for 
delta-discrete $G$-spectra have the transitivity property. To do this, 
we make use of the convention stated below.
\begin{proof}[$\mathbf{Convention \ 5.1}$]
\stepcounter{Lem}
% \begin{Con}\label{convention}
Let $P$ be a discrete group and let $\mathcal{X}$ be a space. By 
\cite[Remark 10.3]{wilkerson}, a {\em proxy action} of $P$ on $\mathcal{X}$ 
is a space $\mathcal{Y}$ that is homotopy equivalent to $\mathcal{X}$ 
and has an action of $P$. Then in \cite[Remark 10.3]{wilkerson}, 
Dwyer and Wilkerson establish
the convention that $\mathcal{X}^{hP}$ is {\em equal} to 
$\mathcal{Y}^{hP}$, and 
a 
proxy action is sometimes referred to as an action. This 
convention 
is an important one: for example, in 
\cite[Section 10]{wilkerson}, this convention plays a role 
in Lemmas 10.4 and 10.6 and in the proof that 
$P$-homotopy fixed points 
have the transitivity property (their Lemma 10.5). 
Thus, in this section, we will make use of the related convention 
described below. 
\par
Let $G$ be any profinite group and let 
$X^{\bullet, \bullet}$ be a bicosimplicial 
discrete $G$-spectrum (that is, 
$X^{\bullet, \bullet}$ is a cosimplicial object 
in $c(\mathrm{Spt}_G)$), 
such that, for all $m, n \geq 0$, $X^{m, n}$ is a fibrant spectrum. 
Let $\{X^{n,n}\}_{[n] \in \Delta}$ be 
the cosimplicial discrete $G$-spectrum that is 
the diagonal 
of $X^{\bullet, \bullet}$; $\{X^{n,n}\}_{[n] \in \Delta}$ is 
defined to be the composition 
\[\Delta \rightarrow \Delta \times \Delta \rightarrow \mathrm{Spt}_G, \ \ \
[n] \mapsto ([n], [n]) \mapsto X^{n,n}.\] 
Then there is a natural 
$G$-equivariant map 
\begin{equation}\label{doubleholim}\zig
\holim_{\Delta \times \Delta} X^{\bullet, \bullet} 
\overset{\simeq}{\longrightarrow} \holim_{[n] \in \Delta} X^{n,n} 
\end{equation} that is a weak equivalence (see, for example, 
\cite[Lemma 5.33]{thomason} and \cite[Remark 19.1.6, Theorem 
19.6.7, (2)]{Hirschhorn}). Notice that the target of 
(\ref{doubleholim}), ${\holim_{[n] \in \Delta} X^{n,n}}$, 
is a delta-discrete $G$-spectrum. Thus, 
we identify the source of 
(\ref{doubleholim}), the $G$-spectrum 
$\holim_{\Delta \times \Delta} X^{\bullet, \bullet},$
with the delta-discrete $G$-spectrum 
${\holim_{[n] \in \Delta} X^{n,n}}$, 
so that
\begin{equation}\label{hdeltag}\zig
(\holim_{\Delta \times \Delta} 
X^{\bullet, \bullet})^{h_\delta G} 
\mathrel{\mathop:}= (\holim_{[n] \in \Delta} X^{n, n})^{h_\delta G}.
\end{equation}
\par
Let $\holim_{\Delta \times \Delta} (X^{\bullet, \bullet})^{hG}$ 
denote $\holim_{([m],[n]) \in \Delta \times \Delta} 
(X^{m,n})^{hG}$ and notice that, 
by Theorem \ref{derived} and Lemma \ref{generalC}, there is a zigzag of 
weak equivalences 
\[\holim_{\Delta \times \Delta} (X^{\bullet, \bullet})^{hG} 
\overset{\simeq}{\longrightarrow} 
\holim_{\Delta \times \Delta} ((X^{\bullet, \bullet})_\fibrant)^G 
\overset{\simeq}{\longleftarrow} 
\bigl(\mathbf{R}(\lim_{\Delta \times \Delta} (-)^G)\bigr)(X^{\bullet, \bullet}).\] 
Hence, it is natural to define the homotopy fixed points of the 
``($\Delta \times \Delta$)-discrete $G$-spectrum" 
$\holim_{\Delta \times \Delta} X^{\bullet, \bullet}$ as 
\[(\holim_{\Delta \times \Delta} X^{\bullet, \bullet})^{hG} 
= \holim_{\Delta \times \Delta} (X^{\bullet, \bullet})^{hG}.\] 
Since each $(X^{m,n})^{hG}$ is a fibrant spectrum, then, 
as in (\ref{doubleholim}), 
there is a weak equivalence
\[(\holim_{\Delta \times \Delta} X^{\bullet, \bullet})^{hG} 
\overset{\simeq}{\longrightarrow} \holim_{[n] \in \Delta} (X^{n,n})^{hG} 
= (\holim_{\Delta \times \Delta} X^{\bullet, \bullet})^{h_\delta G},\] 
which further justifies the convention given in (\ref{hdeltag}). 
\end{proof}
%\begin{flushright} $\qed$ \end{flushright}
%\end{Con}
\par
As in the Introduction, let 
$H$ and $K$ be closed subgroups of $G$, with $H$ normal in $K$. 
Recall from (\ref{identify}) that, given $X \in \cat$, there is an 
identification 
\[X^{hH} = \holim_\Delta \mathrm{Map}_c(G^\bullet, \widehat{X})^H.\] 
Since the map $\psi \: X \rightarrow \widehat{X}$ is natural, it is clear from 
\cite{fibrantmodel} (to be specific, in \cite{fibrantmodel}, 
see pg. 145 and the proof of Theorem 5.2) that the above 
identification is natural in $X$. 
\par
Now let $\holim_\Delta X^\bullet$ be any delta-discrete $G$-spectrum. 
Using the naturality of the above identification, we have 
\[(\holim_\Delta X^\bullet)^{h_\delta H} = 
\holim_{[n] \in \Delta} (X^n)^{hH} = 
\holim_{[n] \in \Delta} \holim_{[m] \in \Delta} 
(\mathrm{Map}_c(G^\bullet, \widehat{X^n})^m)^H.\] Because of the 
isomorphism 
\[\holim_{[n] \in \Delta} \holim_{[m] \in \Delta} 
(\mathrm{Map}_c(G^\bullet, \widehat{X^n})^m)^H \cong 
\holim_{\Delta \times \Delta}
\mathrm{Map}_c(G^\bullet, \widehat{\,X^\bullet\,})^H\] and because 
homotopy limits are ends, which are only unique up to isomorphism, 
we can set 
\[(\holim_\Delta X^\bullet)^{h_\delta H} = 
\holim_{\Delta \times \Delta} \mathrm{Map}_c(G^\bullet, 
\widehat{\,X^\bullet\,})^H.\] 
\par
By Lemma \ref{fibrancy}, for each $m, n \geq 0$, 
$\mathrm{Map}_c(G^\bullet, \widehat{X^n})^m$ is a fibrant 
discrete $H$-spectrum, so that 
$(\mathrm{Map}_c(G^\bullet, \widehat{X^n})^m)^H$ is a fibrant spectrum. 
Also, since 
the diagram $\mathrm{Map}_c(G^\bullet, \widehat{\,X^\bullet\,})$ is a 
bicosimplicial discrete $K$-spectrum, 
$\mathrm{Map}_c(G^\bullet, \widehat{\,X^\bullet\,})^H$ is a 
bicosimplicial discrete $K/H$-spectrum. Thus, 
the discussion above in Convention 5.1 
implies that 
there is a $K/H$-equivariant map 
\[(\holim_\Delta X^\bullet)^{h_\delta H} 
= \holim_{\Delta \times \Delta} \mathrm{Map}_c(G^\bullet, 
\widehat{\,X^\bullet\,})^H
\overset{\simeq}{\longrightarrow} 
\holim_{[n] \in \Delta} 
(\mathrm{Map}_c(G^\bullet, \widehat{X^n})^n)^H\]
that is a weak equivalence, and the target 
of this weak equivalence is a delta-discrete $K/H$-spectrum. 
Therefore, by Convention 5.1, we can 
identify $(\holim_\Delta X^\bullet)^{h_\delta H}$ with the 
delta-discrete $K/H$-spectrum 
$\holim_{[n] \in \Delta} 
(\mathrm{Map}_c(G^\bullet, \widehat{X^n})^n)^H$, 
and hence, by (\ref{hdeltag}) and as in 
the proof of Theorem \ref{steptwo}, 
we have 
\begin{align*}
\bigl((\holim_\Delta X^\bullet)^{h_\delta H}\bigr)^{h_\delta K/H} 
& = \bigl(\holim_{[n] \in \Delta} 
(\mathrm{Map}_c(G^\bullet, \widehat{X^n})^n)^H\bigr)^{h_\delta K/H} \\
& = \holim_{[n] \in \Delta} 
\bigl(\bigl((\mathrm{Map}_c(G^\bullet, 
\widehat{X^n})^n)^H\bigr)_{fK/H}\bigr)^{K/H} \\
& \overset{\simeq}{\longleftarrow} \holim_{[n] \in \Delta} 
\bigl((\mathrm{Map}_c(G^\bullet, 
\widehat{X^n})^n)^H\bigr)^{K/H} \\
& = \holim_{[n] \in \Delta} 
(\mathrm{Map}_c(G^\bullet, 
\widehat{X^n})^n)^K \\ 
& \overset{\simeq}{\longleftarrow} 
\holim_{[n] \in \Delta} \holim_{[m] \in \Delta} 
(\mathrm{Map}_c(G^\bullet, \widehat{X^n})^m)^K \\
& = \holim_{[n] \in \Delta} (X^n)^{hK} \\
& = (\holim_\Delta X^\bullet)^{h_\delta K}.
\end{align*}
\par 
We summarize our work above in the following theorem. 
\begin{Thm}\label{transitivityfordeltadiscretes}
If $G$ has finite vcd and $\holim_\Delta X^\bullet$ 
is a delta-discrete $G$-spectrum, then 
there is a weak equivalence 
\[\bigl((\holim_\Delta X^\bullet)^{h_\delta H}\bigr)^{h_\delta K/H} 
\overset{\simeq}{\longleftarrow} 
(\holim_\Delta X^\bullet)^{h_\delta K}.\]
\end{Thm}
\section{The relationship between $X^{hH}$ and 
$(X_\delta)^{h_\delta H}$, in general}\label{lastsection}
\par
Let $G$ be an arbitrary profinite group and let $X$ be any 
discrete $G$-spectrum. Also, as usual, let $H$ and $K$ be 
closed subgroups of $G$, with $H$ normal in $K$. 
As mentioned near the beginning of the Introduction, 
it is not known, in general, that the ``discrete homotopy 
fixed points" $X^{hH}$ have a $K/H$-action. In 
this section, we consider this 
issue by using the framework of delta-discrete $G$-spectra. 
\par
Given the initial data above, there is a commutative diagram  
\[\xymatrix@+.1in{X_{fH} 
\ar[dr]_-{\Psi_{_{\negthinspace H}}}
& X \ar[l]_-{\mathpzc{f}_{_H}}^-\simeq 
\ar[d]^<<<<<{\underset{U \vartriangleleft_o H} \colim \, \Psi^U} 
\ar[r]^-\Psi_-\simeq & X_\delta\\ 
& \displaystyle{\colim_{U \vartriangleleft_o H}} \, (X_\delta)^U, \ar[ur] &}\] 
where $\colim_{U \vartriangleleft_o H} \Psi^U$, a morphism 
in $\mathrm{Spt}_H$ whose label is a slight abuse of 
notation, is defined to be the composition
\[X \cong \colim_{U \vartriangleleft_o H} X^U \rightarrow 
\colim_{U \vartriangleleft_o H} \, (X_\delta)^U,\] and 
$\Psi_{_{\negthinspace H}}$, a morphism between 
fibrant objects in $\mathrm{Spt}_H$, 
exists because, in 
$\mathrm{Spt}_H$, 
$\mathpzc{f}_{_H}$ is a trivial cofibration and 
$\colim_{U \vartriangleleft_o H} (X_\delta)^U$ is 
fibrant (by \cite[Theorem 3.5]{fibrantmodel}).
\begin{Def}\label{relatingthetwotypes}
The map $\Psi_{_{\negthinspace H}}$ induces the map 
\[X^{hH} = (X_{fH})^H 
%\overset{\scriptstyle{(\Psi_{_{\negthinspace H}})^H}}{\longrightarrow} 
\xrightarrow{\scriptstyle{(\Psi_{_{\negthinspace H}})^H}}
\bigl(\colim_{U \vartriangleleft_o H} \, (X_\delta)^U\bigr)^H 
\cong \holim_\Delta \mathrm{Map}_c(G^\bullet, \widehat{X})^H\] and 
(as in the proof of Lemma \ref{consistency}) there is a weak equivalence
\[\holim_\Delta \mathrm{Map}_c(G^\bullet, \widehat{X})^H 
\overset{\simeq}{\longrightarrow} 
\holim_\Delta \mathrm{Map}_c(G^\bullet, \widehat{X})^{hH} 
= (X_\delta)^{h_\delta H}.\] The composition of these two 
maps defines the map
\[\rho(X)_{_{\negthinspace H}} \: X^{hH} \rightarrow (X_\delta)^{h_\delta H}.\]
% i verified again that it's fine to use \rho above and that it's not used elsewhere
(The map $\rho(X)_{_{\negthinspace H}}$ is not the same as the map 
of Lemma \ref{consistency} 
(when $L = H$), since $\rho(X)_{_{\negthinspace H}}$ 
does not use the identification of (\ref{identify}).)
\end{Def}
\par
Notice that if the map $\colim_{U \vartriangleleft_o H} \Psi^U$ 
is a weak equivalence in $\mathrm{Spt}$, then 
the map $\Psi_{_{\negthinspace H}}$ is a 
weak equivalence in $\mathrm{Spt}_H$, and hence, 
$(\Psi_{_{\negthinspace H}})^H$ is a weak equivalence. Thus, 
if $\colim_{U \vartriangleleft_o H} \Psi^U$ 
is a weak equivalence in $\mathrm{Spt}$, then 
$\rho(X)_{_{\negthinspace H}}$ is a weak equivalence. 
This observation, 
together with \cite[proof of Theorem 4.2]{fibrantmodel} and 
\cite[Proposition 3.3]{mitchell}, immediately yields the following result. 
\begin{Thm}\label{rho}
If $G$ is any profinite group and $X \in \cat$, then 
the map \[\rho(X)_{_{\negthinspace H}} \: 
X^{hH} \overset{\simeq}{\longrightarrow} (X_\delta)^{h_\delta H}\] 
is a weak equivalence, whenever 
any one of the following conditions holds:
\begin{enumerate}
\item[(i)]
$H$ has finite vcd;
\item[(ii)]
$G$ has finite vcd;
\item[(iii)]
there exists a fixed integer $p$ such that $H^s_c(U; \pi_t(X)) = 0$, for all 
$s>p$, all $t \in \mathbb{Z}$, and all $U \vartriangleleft_o H$;  
\item[(iv)]
there exists a fixed integer $q$ such that $H^s_c(U; \pi_t(X)) = 0$, for all 
$t > q$, all $s \geq 0$, and all $U \vartriangleleft_o H$; or
\item[(v)]
there exists a fixed integer $r$ such that $\pi_t(X) = 0$, for all $t > r.$
\end{enumerate} 
\end{Thm} 
\par
In the statement of Theorem \ref{rho}, note that (ii) implies (i) and (v) implies 
(iv). Also, it is not known, in general, that 
$\colim_{U \vartriangleleft_o H} \Psi^U$ is a weak equivalence, 
so that we do not know, in general, 
that $X^{hH}$ and $(X_\delta)^{h_\delta H}$ 
are equivalent. 
\par
As noted in Definition \ref{relatingthetwotypes}, 
$(X_\delta)^{h_\delta H}$ is equivalent to the 
delta-discrete $K/H$-spectrum 
$\holim_\Delta \mathrm{Map}_c(G^\bullet, \widehat{X})^H$, and hence, 
it is natural 
to identify them and to set 
\[\bigl((X_\delta)^{h_\delta H}\bigr)^{h_\delta K/H} = 
\bigl(\holim_\Delta \mathrm{Map}_c(G^\bullet, 
\widehat{X})^H\bigr)^{h_\delta K/H}.\]  
\begin{Thm}\label{deltaiteration}
If $G$ is any profinite group and $X \in \cat$, then 
\[\bigl((X_\delta)^{h_\delta H}\bigr)^{h_\delta K/H} 
\simeq (X_\delta)^{h_\delta K}.\]
\end{Thm}
\begin{proof}
As in the proof of Theorem \ref{steptwo}, 
it is easy to see that there is a zigzag of weak equivalences 
\[\bigl((X_\delta)^{h_\delta H}\bigr)^{h_\delta K/H} 
\overset{\simeq}{\longleftarrow} 
\holim_\Delta \mathrm{Map}_c(G^\bullet, \widehat{X})^K 
\overset{\simeq}{\longrightarrow} (X_\delta)^{h_\delta K}.\]
\end{proof}
\par
Given any profinite group $G$, 
Theorem \ref{deltaiteration} shows that, 
by using $(-)_\delta$ and 
$(-)^{h_\delta H}$, 
delta-discrete homotopy fixed 
points for discrete $G$-spectra are transitive. 
\section{Comparing two different models for the iterated\\homotopy 
fixed point spectrum}
\par
Let $G$ be any profinite group and let $X$ be a discrete $G$-spectrum. 
Recall from \cite[Definition 4.1]{iterated} that $X$ is a hyperfibrant 
discrete $G$-spectrum if the map 
\[\psi(X)^G_L \: (X_{fG})^L \rightarrow
((X_{fG})_{fL})^L\] is a weak 
equivalence for every closed subgroup $L$ of $G$. 
(We use ``$\psi$" in the notation ``$\psi(X)^G_L$," 
because we follow the notation of \cite{iterated}; this use 
of ``$\psi$" is not 
related to the map $\psi$ of Lemma \ref{trivialandhard}.) 
\par
Now suppose that $X$ is a hyperfibrant discrete $G$-spectrum and, as 
usual, let $H$ and $K$ be closed subgroups of $G$, with $H$ 
normal in $K$. These hypotheses imply that 
\begin{itemize} 
\item[(a)]
the map $\psi(X)^G_H \: (X_{fG})^H \overset{\simeq}{\longrightarrow} 
((X_{fG})_{fH})^H$ is a weak equivalence; 
\item[(b)]
the source of the map $\psi(X)^G_H$, the spectrum 
$(X_{fG})^H$, is a discrete 
$K/H$-spectrum; and 
\item[(c)]
since the composition $X \rightarrow
X_{fG} \rightarrow (X_{fG})_{fH}$ is a trivial cofibration 
and the target of the weak equivalence $X \rightarrow
X_{fH}$ is fibrant, in $\mathrm{Spt}_H$, there is a weak equivalence 
$\upsilon \: (X_{fG})_{fH} \rightarrow X_{fH}$ between fibrant 
objects, and hence, there is a weak equivalence $(X_{fG})^H 
\overset{\simeq}{\longrightarrow} X^{hH}$ that is defined by the 
composition
\[\upsilon^H \circ \psi(X)^G_H \: (X_{fG})^H 
\overset{\simeq}{\longrightarrow} ((X_{fG})_{fH})^H 
\overset{\simeq}{\longrightarrow} (X_{fH})^H = X^{hH}.\]
\end{itemize}
Thus, following \cite[Definition 4.5]{iterated}, 
it is natural to define 
\[(X^{hH})^{hK/H} \mathrel{\mathop:}= ((X_{fG})^H)^{hK/H}.\] 
\par
Let $G$ have finite vcd, so that, by Lemma \ref{stepone}, 
$X^{hH} = \holim_\Delta \mathrm{Map}_c(G^\bullet, \widehat{X})^H$ is 
a delta-discrete $K/H$-spectrum. Thus, 
\[(X^{hH})^{h_\delta K/H} = 
\holim_{[n] \in \Delta} \bigl((\mathrm{Map}_c(G^\bullet, 
\widehat{X})^n)^H\bigr)^{hK/H},\] and, by Theorem \ref{steptwo}, 
there is a weak equivalence 
\begin{equation}\label{secondhalf}\zig
\holim_\Delta \mathrm{Map}_c(G^\bullet, \widehat{X})^K 
\overset{\simeq}{\longrightarrow} (X^{hH})^{h_\delta K/H}.
\end{equation}
\par
The above discussion shows that 
when $G$ has finite vcd and $X$ is a hyperfibrant 
discrete $G$-spectrum, there are two different models for the 
iterated homotopy fixed point spectrum, $(X^{hH})^{hK/H}$ and 
$(X^{hH})^{h_\delta K/H}$, and we would like to 
know when they agree with each other. 
The following result gives 
a criterion for when 
$(X^{hH})^{hK/H}$ and $(X^{hH})^{h_\delta K/H}$ are 
equivalent. 
\begin{Thm}\label{comparison}
Let $G$ have finite vcd and suppose that $X$ is a hyperfibrant 
discrete $G$-spectrum. 
If the map $\psi(X)^K_H \: (X_{fK})^H \rightarrow ((X_{fK})_{fH})^H$ is a 
weak equivalence, then there is a weak equivalence
\begin{equation}\label{desired}\zig
(X^{hH})^{hK/H} \overset{\simeq}{\longrightarrow} (X^{hH})^{h_\delta K/H}.
\end{equation}
\end{Thm}
\begin{Rk} Let $G$ and $X$ be as in Theorem \ref{comparison}. 
By definition, if $X$ is a hyperfibrant discrete $K$-spectrum,  
then each map $\psi(X)^K_H$ is a weak equivalence (here, as usual,  
$H$ is any closed subgroup of $G$ that is normal in $K$), giving the 
weak equivalence of (\ref{desired}). 
We refer the reader to 
\cite[Sections 3,$\,$4]{iterated} 
for further discussion about hyperfibrancy. The spectral 
sequence considerations of \cite[pp. \negthinspace 2887--2888]{iterated} 
are helpful for 
understanding when $X$ is a hyperfibrant 
discrete $K$-spectrum.
\end{Rk}
\begin{proof}[Proof of Theorem \ref{comparison}.] 
It follows from Theorem \ref{rho},\,(ii) and Definition \ref{relatingthetwotypes} 
that $(\Psi_{_{\negthinspace K}})^K$ is a weak equivalence, so that, 
by composing with the weak equivalence of (\ref{secondhalf}), 
there is a weak equivalence 
\[X^{hK} \overset{\simeq}{\longrightarrow} 
\holim_\Delta \mathrm{Map}_c(G^\bullet, \widehat{X})^K 
\overset{\simeq}{\longrightarrow} (X^{hH})^{h_\delta K/H}.\] 
Therefore, to 
obtain (\ref{desired}), it suffices to show that 
there is a weak equivalence 
\[(X^{hH})^{hK/H} = ((X_{fG})^H)^{hK/H} 
\rightarrow X^{hK}.\]
\par
It is useful for the argument below 
to recall from the proof of \cite[Lemma 4.6]{iterated} 
that the functor $(-)^H \: \mathrm{Spt}_K \rightarrow \mathrm{Spt}_{K/H}$ 
is a right Quillen functor. 
\par
% i verified that \eta is never used elsewhere in this paper
Since $X \rightarrow X_{fG}$ is a trivial cofibration in 
$\mathrm{Spt}_K$ and the map $X \rightarrow 
X_{fK}$ has a fibrant target, there is a weak equivalence 
\[\mu \: X_{fG} \overset{\simeq}{\longrightarrow} X_{fK}\] in 
$\mathrm{Spt}_K$. 
Thus, there is a commutative diagram 
\[\xymatrix@+.1in{((X_{fG})^H)_{fK/H} \ar[dr]_{\widetilde{\,\mu^H\,}} 
& (X_{fG})^H \ar[l]^-{\simeq}_-{\eta}  
\ar[d]^-{\mu^H} 
\ar[r]^-{\psi(X)^G_H}_-{\simeq} & ((X_{fG})_{fH})^H 
\ar[d]^-{(\mu_{_{fH}})^H}_-{\simeq} \\
& (X_{fK})^H \ar[r]^-{\psi(X)^K_H} & 
((X_{fK})_{fH})^H.}\] 
The map $\widetilde{\,\mu^H\,}$ 
exists in $\mathrm{Spt}_{K/H}$, because in $\mathrm{Spt}_{K/H}$ 
the map $\eta$ is a trivial 
cofibration and $(X_{fK})^H$ is fibrant. Also, as noted in the diagram, 
the map $\psi(X)^G_H$ is a 
weak equivalence, 
since $X$ is a hyperfibrant discrete $G$-spectrum, and 
the rightmost vertical map, $(\mu_{fH})^H = \mu^{hH}$, 
is a weak equivalence, 
since $\mu$ is a weak equivalence in 
$\mathrm{Spt}_H$.
% i verified that \mu appears nowhere else in this paper and so 
%	it's fine to use it here
%
% i have verified twice that the discussion i'm giving here 
% 	does *NOT* appear in my iterated paper, and i don't need to 
%	check for this anymore
\par
Now suppose that $\psi(X)^K_H$ is a weak equivalence. 
Then $\mu^H$ is a weak equivalence, and hence, 
$\widetilde{\,\mu^H\,}$ is a weak equivalence 
between fibrant objects in $\mathrm{Spt}_{K/H}$. Therefore, there is 
a weak equivalence 
\[{\Bigl(\widetilde{\,\mu^H\,}\Bigr)}^{K/H} \: 
((X_{fG})^H)^{hK/H} = (((X_{fG})^H)_{fK/H})^{K/H} 
\overset{\simeq}{\longrightarrow} ((X_{fK})^H)^{K/H} = X^{hK},\] 
completing the proof. 
% rho appears nowhere in paper, so it's fine notation here 
\end{proof}
\section{Delta-discrete homotopy fixed points are 
always\\ discrete homotopy fixed points}\label{modelcategories}
Let $G$ be any profinite group. 
In this final section, we obtain in two different ways 
the conclusion stated in the section title. Somewhat 
interestingly, the discrete homotopy fixed points, 
in both cases, are those of a discrete $G$-spectrum 
that, in general, need not be equivalent to the delta-discrete $G$-spectrum 
(whose delta-discrete homotopy fixed points are under consideration).
\par
Let $\holim_\Delta X^\bullet$ be any delta-discrete $G$-spectrum 
and define 
\[C(X^\bullet) := \colim_{N \vartriangleleft_o G} \, 
\bigl(\holim_{[n] \in \Delta} (X^n)_{fG}\bigr)^N.\] Notice that 
$C(X^\bullet)$ is a discrete $G$-spectrum and 
\[(\holim_\Delta X^\bullet)^{h_\delta G} \cong 
\bigl(\holim_{[n] \in \Delta} (X^n)_{fG}\bigr)^G 
\cong (C(X^\bullet))^G \overset{\simeq}{\longrightarrow} 
(C(X^\bullet))^{hG},\] where the justification 
for the second isomorphism above 
is as in \cite[proof of Theorem 2.3]{fibrantmodel} 
and the weak equivalence is due to the fact that the spectrum
$C(X^\bullet)$ is a fibrant discrete 
$G$-spectrum (by \cite[Corollary 2.4]{fibrantmodel}). Also, notice that 
the weak equivalence 
\begin{equation}\label{equivalenceone}\zig 
(\holim_\Delta X^\bullet)^{h_\delta G} \overset{\simeq}{\longrightarrow} 
(C(X^\bullet))^{hG},
\end{equation} defined above, 
is partly induced by the canonical 
map $\iota_{_G}$ (defined by a colimit of inclusions 
of fixed points) in the
zigzag 
\begin{equation}\label{zigzagone}\zig
C(X^\bullet) \overset{\iota_{_G}}{\longrightarrow} 
\holim_{[n] \in \Delta} (X^n)_{fG} 
\overset{\simeq}{\longleftarrow} \holim_\Delta X^\bullet
\end{equation} of $G$-equivariant maps. In (\ref{zigzagone}), 
the second map, as indicated, 
is a weak equivalence of delta-discrete $G$-spectra.
\par
Though (\ref{equivalenceone}) shows that the delta-discrete 
homotopy fixed points of any delta-discrete $G$-spectrum can 
be realized as the discrete homotopy fixed points of a discrete 
$G$-spectrum, there is a slight incongruity here: the map $\iota_{_G}$ 
in (\ref{zigzagone}) does not have to be a weak equivalence. For 
example, if $\iota_{_G}$ is a weak equivalence, then 
$\pi_0(\holim_\Delta X^\bullet)$ is a discrete $G$-module 
(since $\pi_0(C(X^\bullet))$ is a discrete $G$-module), but 
this is not the case, for example, when $\holim_\Delta X^\bullet$ 
is the delta-discrete $\mathbb{Z}_q$-spectrum $Y^{h\mathbb{Z}/p}$ 
(this characterization of $Y^{h\mathbb{Z}/p}$ is obtained by applying 
Lemma \ref{stepone}) 
that was referred to in the Introduction, since $\pi_0(Y^{h\mathbb{Z}/p})$ 
is not a discrete $\mathbb{Z}_q$-module (by \cite[Appendix A]{iterated}). 
\par
Now we consider a second and more interesting way to realize 
$(\holim_\Delta X^\bullet)^{h_\delta G}$ as a discrete homotopy 
fixed point spectrum. As in Section \ref{properties}, let $(X^\bullet)_\fibrant$ 
denote the fibrant replacement of $X^\bullet$ in the model category 
$c(\cat)$. Then there is again a zigzag 
\begin{equation}\label{zigzagtwo}\zig
\lim_\Delta (X^\bullet)_\fibrant 
\overset{\lambda}{\longrightarrow}
\holim_\Delta (X^\bullet)_\fibrant 
\underset{\overset{\scriptstyle{\gamma}}{}}{\overset{\simeq}{\longleftarrow}}
\holim_\Delta X^\bullet 
\end{equation} of canonical 
$G$-equivariant maps, where $\lambda$ is the usual 
map in $\mathrm{Spt}$ 
from the limit to the homotopy limit and 
the map 
$\gamma$ is a weak equivalence of delta-discrete $G$-spectra. 
We will show that, as in the case of zigzag 
(\ref{zigzagone}), $\lim_\Delta 
(X^\bullet)_\fibrant$ is a discrete $G$-spectrum and its 
discrete homotopy 
fixed points are 
equivalent to $(\holim_\Delta X^\bullet)^{h_\delta G}$, 
but, as before, the map $\lambda$ need not be a weak equivalence. 
\begin{Lem}\label{finallemmaone}
Let $\smash{\displaystyle{\holim_\Delta}} \, 
X^\bullet$ be a delta-discrete $G$-spectrum. 
Then $\smash{\displaystyle{\lim_\Delta}} \, 
(X^\bullet)_\fibrant$ is a discrete $G$-spectrum.
\end{Lem}
\begin{proof}
Given two morphisms 
$\negthinspace \xymatrix@C=15pt{Y_0 \ar@<.8ex>[r] 
\ar@<-.6ex>[r] & Y_1}\negthinspace$ in $\cat$, 
let $\mathrm{equal}[\xymatrix@C=15pt{Y_0 \ar@<.8ex>[r] 
\ar@<-.6ex>[r] & Y_1}\negthinspace]$ 
denote the equalizer in $\mathrm{Spt}$.
Also, let 
$\mathrm{equal}_G[\xymatrix@C=15pt{Y_0 \ar@<.8ex>[r] 
\ar@<-.6ex>[r] & Y_1}\negthinspace]$ 
denote the equalizer in $\cat$. 
Due to the fact that 
$X^\bullet$ and $(X^\bullet)_\fibrant$ are cosimplicial 
discrete $G$-spectra, to prove the lemma it 
suffices to show that 
$\lim_\Delta X^\bullet$ is a discrete $G$-spectrum. 
\par
Since limits in 
$\mathrm{Spt}$ are formed levelwise 
in simplicial sets and, 
given a cosimplicial simplicial set $Z^\bullet$, there is a natural 
isomorphism $\lim_\Delta Z^\bullet \overset{\cong}{\longrightarrow} 
Z^{-1}$, where 
$Z^{-1}$ is the equalizer of the coface maps 
$\negthinspace\negthinspace\xymatrix@C=20pt{Z^0 
\ar@<.6ex>[r]^-{\scriptscriptstyle{d^0}} 
\ar@<-.8ex>[r]_-{\scriptscriptstyle{d^1}} & Z^1 \negthinspace\negthinspace}$ 
(see, for example, 
\cite[Lemma 1]{cosimp3}), it follows that the 
canonical $G$-equivariant map 
\[\displaystyle{\lim_\Delta X^\bullet \overset{\cong}{\longrightarrow} 
\mathrm{equal}\Bigl[\negthinspace\xymatrix{X^0
\ar@<.6ex>[r]^-{d^0} 
\ar@<-.8ex>[r]_-{d^1} & X^1}\negthinspace\Bigr]}\] 
is an isomorphism in $\mathrm{Spt}$. 
\par
The proof is completed by noting that there are isomorphisms
\begin{align*}
\mathrm{equal}_G\Bigl[\negthinspace\xymatrix{X^0 
\ar@<.6ex>[r]^-{d^0} 
\ar@<-.8ex>[r]_-{d^1} & X^1}\negthinspace \Bigr] & \cong 
\colim_{N \vartriangleleft_o G} 
\, \Bigl(\mathrm{equal}
\Bigl[\negthinspace\xymatrix{X^0 
\ar@<.6ex>[r]^-{d^0} 
\ar@<-.8ex>[r]_-{d^1} & 
X^1}\negthinspace\Bigr]
\Bigr)^{\negthinspace N} \\
& \cong \mathrm{equal}\biggl[\, \colim_{N \vartriangleleft_o G} 
\negthinspace \negthinspace 
\xymatrix{(X^0)^N
\ar@<.6ex>[r]^-{d^0} 
\ar@<-.8ex>[r]_-{d^1} &{}} \negthinspace \negthinspace 
\colim_{N \vartriangleleft_o G} 
(X^1)^N\biggr] \\
& \cong 
\mathrm{equal}\Bigl[\negthinspace\xymatrix{X^0 
\ar@<.6ex>[r]^-{d^0} 
\ar@<-.8ex>[r]_-{d^1} & 
X^1}\negthinspace\Bigr] \\
& \cong \lim_\Delta X^\bullet,
\end{align*} 
with each isomorphism $G$-equivariant. 
In this string 
of isomorphisms, the first one is because of \cite[Remark 4.2]{cts} 
and the third one follows from the fact that 
$X^0$ and $X^1$ 
are discrete $G$-spectra.
\end{proof} 
\begin{Lem}\label{finallemmatwo}
Let $\smash{\displaystyle{\holim_\Delta}} 
\, X^\bullet$ be a delta-discrete $G$-spectrum. 
Then there is an equivalence 
\[(\lim_\Delta (X^\bullet)_\fibrant)^{hG} \simeq 
(\holim_\Delta X^\bullet)^{h_\delta G}.\]
\end{Lem}
\begin{proof} 
Let $\displaystyle{\lim_\Delta}^G \,Y^\bullet$ denote the limit in $\cat$ of an 
object $Y^\bullet \in c(\cat)$. 
Since $\displaystyle{\lim_\Delta} (X^\bullet)_\fibrant \in \cat$, there 
are isomorphisms
\begin{equation}\label{thereareisos}\zig
\lim_\Delta (X^\bullet)_\fibrant \cong \colim_{N \vartriangleleft_o G} 
\bigl(\lim_\Delta (X^\bullet)_\fibrant\bigr)^N \cong 
{\lim_\Delta}^G (X^\bullet)_\fibrant
\end{equation}
in $\cat$. 
Since the functor $\cat \rightarrow c(\cat)$ that sends a discrete 
$G$-spectrum $Y$ to $\mathrm{cc}^\bullet(Y)$ 
(equipped with its natural $G$-action) preserves weak equivalences 
and cofibrations, its right adjoint 
${\displaystyle{\lim_\Delta}}^G (-) \: c(\cat) \rightarrow 
\cat$ preserves fibrant objects, and hence, 
$\displaystyle{\lim_\Delta}^G (X^\bullet)_\fibrant$ is a fibrant 
discrete $G$-spectrum, and hence, the canonical map 
\[({\lim_\Delta}^G (X^\bullet)_\fibrant)^G 
\overset{\simeq}\longrightarrow 
({\lim_\Delta}^G (X^\bullet)_\fibrant)^{hG} \cong 
(\lim_\Delta (X^\bullet)_\fibrant)^{hG}\] is 
a weak equivalence. 
\par
The proof is completed by the zigzag of equivalences
\[({\lim_\Delta}^G (X^\bullet)_\fibrant)^G 
\cong 
\lim_\Delta ((X^\bullet)_\fibrant)^G \overset{\simeq}{\longrightarrow} 
\holim_\Delta ((X^\bullet)_\fibrant)^G 
\overset{\simeq}{\longleftarrow} (\holim_\Delta X^\bullet)^{h_\delta G},\]
where the isomorphism uses (\ref{thereareisos}), the 
first weak equivalence is by Theorem \ref{derived}, and 
the second weak equivalence applies Theorem \ref{derivedfunctor}.
\end{proof}
\par
The argument that was used earlier to show that $\iota_{_G}$ in 
zigzag (\ref{zigzagone}) does not have 
to be a weak equivalence also applies to show that 
the map $\lambda$ in zigzag (\ref{zigzagtwo}) 
does not have to be a weak equivalence.
\begin{Rk}\label{forgetfulfunctor}
Notice that if $(X^\bullet)_\fibrant$ is fibrant in $c(\mathrm{Spt})$ 
(equipped with the injective model structure), then, 
as in the proof of Theorem \ref{derived}, the map $\lambda$ is a 
weak equivalence. Thus, interestingly, 
fibrations in $c(\cat)$ are not necessarily 
fibrations in $c(\mathrm{Spt})$ -- even though any fibration in $\cat$ 
is a fibration in $\mathrm{Spt}$, by \cite[Lemma 3.10]{cts}. 
This observation is closely related to the fact
that the forgetful functor 
$U \: \cat \rightarrow \mathrm{Spt}$ need not be a right 
adjoint (see \cite[Section 3.6]{joint} for a discussion of this fact): 
if one supposes that $U$ is a right adjoint, then $U$ is also 
a right Quillen functor, since 
$U$ preserves fibrations and weak equivalences, 
and hence, by 
\cite[Remark A.2.8.6]{luriebook}, the forgetful functor 
$U \circ (-) \: c(\cat) \rightarrow c(\mathrm{Spt})$ 
preserves fibrations.
\par
Let $V$ be an open subgroup of $G$. 
By \cite[proof of Lemma 3.1]{iterated}, the right adjoint 
$\mathrm{Res}_G^V \: \cat \rightarrow \mathrm{Spt}_V$ that 
regards a discrete $G$-spectrum as a discrete $V$-spectrum 
is a right Quillen functor, and hence (again, by \cite[Remark 
A.2.8.6]{luriebook}), the restriction functor $\mathrm{Res}_G^V 
\circ (-) \: c(\cat) \rightarrow c(\mathrm{Spt}_V)$ preserves fibrations. 
Thus, \[(X^\bullet)_\fibrant \cong \colim_{N \vartriangleleft_o G} 
((X^\bullet)_\fibrant)^N\] is the filtered colimit of cosimplicial spectra 
$((X^\bullet)_\fibrant)^N$, each of which is fibrant in $c(\mathrm{Spt})$, 
since $\mathrm{Res}_G^N \circ (X^\bullet)_\fibrant = 
(X^\bullet)_\fibrant$ is fibrant in $c(\mathrm{Spt}_N)$ and, as in 
the proof of Theorem \ref{derived}, 
the functor $(-)^N \: c(\mathrm{Spt}_N) \rightarrow c(\mathrm{Spt})$ 
preserves fibrant objects. Since there are cases 
where $(X^\bullet)_\fibrant$ is 
not fibrant in $c(\mathrm{Spt})$, we can conclude, 
perhaps somewhat surprisingly, that, unlike in $\mathrm{Spt}$, 
a filtered colimit of fibrant objects in $c(\mathrm{Spt})$ does not 
have to be 
fibrant. This also shows that, though $c(\mathrm{Spt})$ 
is cofibrantly generated (by \cite[Proposition~A.2.8.2]{luriebook}), 
it is not weakly finitely generated (see 
\cite[Definition~3.4,~Lemma 3.5]{Dundascolim}). 
\end{Rk}
\par
A priori, we do not expect a delta-discrete $G$-spectrum, in general, 
to be equivalent to a discrete $G$-spectrum, and the fact that 
each of 
two natural ways to realize an arbitrary delta-discrete homotopy 
fixed point spectrum as a discrete homotopy fixed point spectrum fails 
to come from such a general equivalence is consistent 
with our expectation.
\par
The following result shows 
that zigzags (\ref{zigzagone}) and (\ref{zigzagtwo}) 
are, in fact, directly related to each other (beyond just having the 
``same general structure").
\begin{Thm}\label{finallemmathree}
Let $\holim_\Delta X^\bullet$ be any delta-discrete $G$-spectrum. 
Then the map $\iota_{_G}$ in (\ref{zigzagone}) is a weak equivalence 
if and only if the map $\lambda$ in (\ref{zigzagtwo}) is a weak 
equivalence. 
\end{Thm}
\begin{proof}
Let $\holim_\Delta^G Y^\bullet$ denote the 
homotopy limit in $\cat$ of $Y^\bullet$, a cosimplicial discrete 
$G$-spectrum. Then, by \cite[Theorem 2.3]{fibrantmodel} and because 
homotopy limits are ends and thereby 
only unique up to isomorphism, 
there is the identity
\[{\holim_\Delta}^G \, Y^\bullet = \colim_{N \vartriangleleft_o G} 
(\holim_\Delta Y^\bullet)^N,\] which implies 
that $C(X^\bullet) = \holim_{[n] \in \Delta}^G (X^n)_{fG}.$ 
Similarly, there is the 
identity \[{\lim_\Delta}^G (X^\bullet)_\fibrant = 
\colim_{N \vartriangleleft_o G} (\lim_\Delta (X^\bullet)_\fibrant)^N.\] 
Also, 
since the map 
$\{X^n\}_{[n] \in \Delta} \rightarrow \{(X^n)_{fG}\}_{[n] \in \Delta}$ is a 
trivial cofibration in $c(\cat)$, there is a 
weak equivalence \[\widetilde{\gamma} \: 
\{(X^n)_{fG}\}_{[n] \in \Delta} \rightarrow 
(X^\bullet)_\fibrant\] in $c(\cat)$.
\par
Given the above facts, there is the commutative diagram 
\[\xymatrix@C=17pt@R=22pt{\smash{\displaystyle{\lim_\Delta}} \, 
(X^\bullet)_\fibrant 
\ar[r]^-\lambda & 
\smash{\displaystyle{\holim_\Delta}} \, 
(X^\bullet)_\fibrant & 
\smash{\displaystyle{\holim_{[n] \in \Delta}}} \, (X^n)_{fG} 
\ar[l]_-\simeq^-{\widehat{ \ \widetilde{\gamma} \ }} 
\\ \smash{\displaystyle{\lim_\Delta}}^G 
(X^\bullet)_\fibrant \ar[u]^-\cong 
\ar[r]^-\simeq_-{\lambda'_G} & 
\smash{\displaystyle{\holim_\Delta}}^G 
(X^\bullet)_\fibrant \ar[u] 
& \smash{\displaystyle{\holim_{[n] \in \Delta}}}^G 
(X^n)_{fG} \ar[l]_-\simeq^-{\widehat{\ \widetilde{\gamma} 
\ }_{\negthinspace \negthinspace \negthinspace G}} 
\ar[u]_-{\iota_{_G}}}\] 
of canonical maps: 
each vertical map is induced by inclusions of fixed points; 
the map $\widehat{ \ \widetilde{\gamma} \ }$ is a weak equivalence (of 
delta-discrete $G$-spectra) because, in $c(\mathrm{Spt})$, 
$\widetilde{\gamma}$ is an 
objectwise weak equivalence between objectwise fibrant 
diagrams; the leftmost vertical map is an 
isomorphism thanks to Lemma \ref{finallemmaone}; 
the proof of Lemma \ref{injectivebehavior} below implies that 
$\lambda'_G$ is a weak equivalence; 
and, in $c(\cat)$, the map $\widetilde{\gamma}$ is an 
objectwise weak equivalence between objectwise fibrant 
diagrams, so that, by 
\cite[Theorem 18.5.3, (2)]{Hirschhorn}, 
$\widehat{\ \widetilde{\gamma} 
\ }_{\negthinspace \negthinspace \negthinspace G}$ 
is a weak equivalence.
\par
The desired conclusion follows immediately from this diagram.
\end{proof}
\par
Lemma \ref{injectivebehavior} below, whose justification is used in 
the proof of 
the previous result, is an example of the statement 
that if $\mathcal{C}$ is a small category and $\mathcal{M}$ is a 
combinatorial simplicial model category, 
with $M$ a fibrant object in 
the category $\mathcal{M}^\mathcal{C}$ of $\mathcal{C}$-shaped 
diagrams in $\mathcal{M}$, when $\mathcal{M}^\mathcal{C}$ is 
equipped with the injective model 
structure, then the canonical morphism $\lim_\mathcal{C} M \rightarrow 
\holim_\mathcal{C} M$ (defined as in 
\cite[Example 18.3.8,~(2)]{Hirschhorn}) 
is a weak equivalence in $\mathcal{M}.$ However, 
since the author is not able to point to a place in the 
literature that has an explicit proof of 
this particular statement, we give the details for the 
case that we need.
\begin{Rk}\label{holimremark}
The author would like to qualify his comment related to 
the literature on homotopy limits that is in the last 
sentence above. Let $\mathcal{N}$ 
be a simplicial model category such that, for some 
small category $\mathcal{D}$, the diagram 
category $\mathcal{N}^\mathcal{D}$ has an 
injective 
model structure, with $N$ fibrant in $\mathcal{N}^\mathcal{D}$. 
Then the statement 
that ``the canonical morphism 
$\lim_\mathcal{D} N \rightarrow \holim_\mathcal{D} N$ is a weak 
equivalence in $\mathcal{N} \, "$ should follow in an essentially formal 
way from \cite[Corollary 18.4.2,~(1)]{Hirschhorn} as its ``model 
category dual," but verifying 
this requires some care and the author has 
not checked all the details (when doing this, an application 
of \cite[Corollary 18.4.5,~(2)]{Hirschhorn} is helpful).
\par
Also, to place the above statement in a more 
general and useful framework, see \cite{DHKS}. In the 
dual case of the homotopy colimit, \cite[Sections~5-9,~13]{shulman} 
is helpful for relating \cite{DHKS} to the dual 
of the above statement. A helpful perspective 
is given in the discussion of homotopy right Kan extensions 
and homotopy limits in 
\cite[Section A.2.8; e.g., Remark A.2.8.8]{luriebook}.
\end{Rk}
\begin{Lem}\label{injectivebehavior}
Let $G$ be any profinite group and $\mathcal{C}$ any 
small category. If $Y^{\typical}$ is fibrant in $(\cat)^\mathcal{C}$, 
when $(\cat)^\mathcal{C}$ is 
equipped with the injective model structure, then the canonical 
map \[\lambda_G \: 
{\lim_\mathcal{C}}^G \, Y^{\typical} \overset{\simeq}{\longrightarrow}
{\holim_\mathcal{C}}^G \, Y^{\typical}\] is a weak equivalence in $\cat$. 
(Here, like before, 
${\displaystyle{\lim_\mathcal{C}}}^G \, Y^{\typical}$ and 
${\displaystyle{\holim_\mathcal{C}}}^G \, Y^{\typical}$ 
denote the limit and 
homotopy limit, respectively, in $\cat$ of $Y^{\typical}$.)
\end{Lem} 
\begin{proof}
By the initial comments in the proof of Theorem 
\ref{finallemmathree}, we can identify the map $\lambda_G$ with the 
canonical morphism
\[\lambda'_G \: 
\colim_{N \vartriangleleft_o G} (\lim_{\superc \in \mathcal{C}} Y^{\superc})^N 
\rightarrow \colim_{N \vartriangleleft_o G} 
(\holim_{\superc \in \mathcal{C}} Y^{\superc})^N\] in $\cat$, and thus, 
to obtain the desired conclusion, 
it suffices to show that the map 
\[\lambda''_G \: 
\colim_{N \vartriangleleft_o G} \, \lim_{\superc \in \mathcal{C}} 
(Y^{\superc})^N 
\rightarrow \colim_{N \vartriangleleft_o G} 
\, \holim_{\superc \in \mathcal{C}} (Y^{\superc})^N\] 
is a weak equivalence in $\mathrm{Spt}$.
\par
As in 
Remark \ref{forgetfulfunctor}, $\{(Y^{\superc})^N\}_{\superc \in \mathcal{C}}$ 
is fibrant in 
$\mathrm{Spt}^\mathcal{C}$, for each $N$, so that 
each 
$\lim_{\superc \in \mathcal{C}} (Y^{\superc})^N$ is fibrant in $\mathrm{Spt}$ 
and every $(Y^{\superc})^N$ is fibrant in $\mathrm{Spt}$. This last 
conclusion implies that 
each $\holim_{\superc \in \mathcal{C}} (Y^{\superc})^N$ 
is a fibrant spectrum. 
Therefore, $\lambda''_G$ is the filtered colimit of the maps 
\[\lambda^N \: \lim_{\superc \in \mathcal{C}} (Y^{\superc})^N 
\rightarrow \holim_{\superc \in \mathcal{C}} (Y^{\superc})^N,\] 
each of which is a map between 
fibrant spectra, and hence, to show that 
$\lambda''_G$ is a weak equivalence, we only have to show that 
each $\lambda^N$ is a weak equivalence. Since 
$\{(Y^{\superc})^N\}_{\superc \in \mathcal{C}}$ is fibrant in 
$\mathrm{Spt}^\mathcal{C}$, the main argument in the 
proof of Theorem \ref{derived} implies that $\lambda^N$ is a 
weak equivalence, completing the proof.
\end{proof}

%\bibliographystyle{plain}
%\bibliography{biblio} 

\begin{thebibliography}{10}

\bibitem{joint}
Mark Behrens and Daniel~G. Davis.
\newblock The homotopy fixed point spectra of profinite {G}alois extensions.
\newblock {\em Trans. Amer. Math. Soc.}, 362(9):4983--5042, 2010.
 
\bibitem{Borceux}
Francis Borceux.
\newblock {\em Handbook of categorical algebra, 3: {C}ategories of sheaves},
  volume~52 of {\em Encyclopedia of Mathematics and its Applications}.
\newblock Cambridge University Press, Cambridge, 1994.

\bibitem{cts}
Daniel~G. Davis.
\newblock Homotopy fixed points for {$L\sb {K(n)}(E\sb n\wedge X)$} using the
  continuous action.
\newblock {\em J. Pure Appl. Algebra}, 206(3):322--354, 2006.

\bibitem{fibrantmodel}
Daniel~G. Davis.
\newblock Explicit fibrant replacement for discrete {$G$}-spectra.
\newblock {\em Homology, Homotopy Appl.}, 10(3):137--150, 2008.

\bibitem{iterated}
Daniel~G. Davis.
\newblock Iterated homotopy fixed points for the {L}ubin-{T}ate spectrum
  \textrm{(with an appendix by {$\mathrm{D}$}.{$\mathrm{G}$}.
  {$\mathrm{D}$}avis and {$\mathrm{B}$}. {$\mathrm{W}$}ieland)}.
\newblock {\em Topology Appl.}, 156(17):2881--2898, 2009.

\bibitem{LHS}
Ethan~S. Devinatz.
\newblock A {L}yndon-{H}ochschild-{S}erre spectral sequence for certain
  homotopy fixed point spectra.
\newblock {\em Trans. Amer. Math. Soc.}, 357(1):129--150 (electronic), 2005.

\bibitem{DH}
Ethan~S. Devinatz and Michael~J. Hopkins.
\newblock Homotopy fixed point spectra for closed subgroups of the {M}orava
  stabilizer groups.
\newblock {\em Topology}, 43(1):1--47, 2004.

\bibitem{dikran}
Dikran Dikranjan.
\newblock Topological characterization of {$p$}-adic numbers and an application
  to minimal {G}alois extensions.
\newblock {\em Annuaire Univ. Sofia Fac. Math. M\'ec.}, 73:103--110, 1979.

\bibitem{combinatorial}
Daniel Dugger.
\newblock Combinatorial model categories have presentations.
\newblock {\em Adv. Math.}, 164(1):177--201, 2001.

\bibitem{Dundascolim}
Bj\o rn Ian Dundas, Oliver R\"{o}ndigs, and Paul Arne \O stv\ae r.
\newblock Enriched functors and stable homotopy theory.
\newblock {\em Doc. Math.}, 8:409--488 (electronic), 2003.

\bibitem{wilkerson}
W.~G. Dwyer and C.~W. Wilkerson.
\newblock Homotopy fixed-point methods for {L}ie groups and finite loop spaces.
\newblock {\em Ann. of Math. (2)}, 139(2):395--442, 1994.

\bibitem{neisen}
William Dwyer, Haynes Miller, and Joseph Neisendorfer.
\newblock Fibrewise completion and unstable {A}dams spectral sequences.
\newblock {\em Israel J. Math.}, 66(1-3):160--178, 1989.

\bibitem{DHKS}
William~G. Dwyer, Philip~S. Hirschhorn, Daniel~M. Kan, and Jeffrey~H. Smith.
\newblock {\em {Homotopy limit functors on model categories and homotopical
  categories.}}
\newblock {Providence, RI: American Mathematical Society (AMS)}, 2004.

\bibitem{fausk}
Halvard Fausk.
\newblock Equivariant homotopy theory for pro-spectra.
\newblock {\em Geom. Topol.}, 12(1):103--176, 2008.

\bibitem{Pgg/Hop0}
P.~G. Goerss and M.~J. Hopkins.
\newblock Moduli spaces of commutative ring spectra.
\newblock In {\em Structured ring spectra}, volume 315 of {\em London Math.
  Soc. Lecture Note Ser.}, pages 151--200. Cambridge Univ. Press, Cambridge,
  2004.

\bibitem{GJ}
Paul~G. Goerss and John~F. Jardine.
\newblock {\em Simplicial homotopy theory}.
\newblock Birkh\"auser Verlag, Basel, 1999.

\bibitem{grigorchuk}
R.~I. Grigorchuk.
\newblock Just infinite branch groups.
\newblock In {\em New horizons in pro-{$p$} groups}, volume 184 of {\em Progr.
  Math.}, pages 121--179. Birkh\"auser Boston, Boston, MA, 2000.

\bibitem{galoishess}
Kathryn Hess.
\newblock A general framework for homotopic descent and codescent.
\newblock Preprint, 68 pp., arXiv:1001.1556, May 27, 2010.

\bibitem{Hirschhorn}
Philip~S. Hirschhorn.
\newblock {\em Model categories and their localizations}, volume~99 of {\em
  Mathematical Surveys and Monographs}.
\newblock American Mathematical Society, Providence, RI, 2003.

\bibitem{hoveygeneral}
Mark Hovey.
\newblock Spectra and symmetric spectra in general model categories.
\newblock {\em J. Pure Appl. Algebra}, 165(1):63--127, 2001.

\bibitem{Jardinecanada}
J.~F. Jardine.
\newblock Stable homotopy theory of simplicial presheaves.
\newblock {\em Canad. J. Math.}, 39(3):733--747, 1987.

\bibitem{Jardine}
J.~F. Jardine.
\newblock {\em Generalized \'etale cohomology theories}.
\newblock Birkh\"auser Verlag, Basel, 1997.

\bibitem{cosimp3}
J.F. Jardine.
\newblock Cosimplicial spaces and cocycles.
\newblock Preprint, 26 pp., April 19, 2010.

\bibitem{luriebook}
Jacob Lurie.
\newblock {\em Higher topos theory}, volume 170 of {\em Annals of Mathematics
  Studies}.
\newblock Princeton University Press, Princeton, NJ, 2009.

\bibitem{mitchell}
Stephen~A. Mitchell.
\newblock Hypercohomology spectra and {T}homason's descent theorem.
\newblock In {\em Algebraic $K$-theory (Toronto, ON, 1996)}, pages 221--277.
  Amer. Math. Soc., Providence, RI, 1997.

\bibitem{oates}
Sidney~A. Morris, Sheila Oates-Williams, and H.~B. Thompson.
\newblock Locally compact groups with every closed subgroup of finite index.
\newblock {\em Bull. London Math. Soc.}, 22(4):359--361, 1990.

\bibitem{quicklubin}
Gereon Quick.
\newblock Homotopy fixed points for {L}ubin-{T}ate spectra.
\newblock Preprint, 17 pp., Dec. 22, 2009, arXiv:0911.5238.

\bibitem{Rognes}
John Rognes.
\newblock Galois extensions of structured ring spectra.
\newblock In {\em Galois extensions of structured ring spectra/{S}tably
  dualizable groups, {$\mathrm{in}$} Mem. Amer. Math. Soc., {$\mathrm{vol. \
  192 \ (898)}$}}, pages 1--97, 2008.

\bibitem{rosickybrown}
Ji{\v{r}}{\'{\i}} Rosick{\'y}.
\newblock Generalized {B}rown representability in homotopy categories.
\newblock {\em Theory Appl. Categ.}, 14:no. 19, 451--479 (electronic), 2005.

\bibitem{shulman}
Michael Shulman.
\newblock Homotopy limits and colimits and enriched homotopy theory.
\newblock Preprint, 79 pp., arXiv:math/0610194, July 1, 2009.

\bibitem{thomason}
R.~W. Thomason.
\newblock Algebraic ${K}$-theory and \'etale cohomology.
\newblock {\em Ann. Sci. \'Ecole Norm. Sup. (4)}, 18(3):437--552, 1985.

\bibitem{hagtwo}
Bertrand To{\"e}n and Gabriele Vezzosi.
\newblock Homotopical algebraic geometry, {II}: {G}eometric stacks and
  applications.
\newblock {\em Mem. Amer. Math. Soc.}, 193(902):x+224, 2008.

\bibitem{justinfinite}
John~S. Wilson.
\newblock On just infinite abstract and profinite groups.
\newblock In {\em New horizons in pro-{$p$} groups}, volume 184 of {\em Progr.
  Math.}, pages 181--203. Birkh\"auser Boston, Boston, MA, 2000.

\end{thebibliography}

\end{document}